\documentclass[12pt]{amsart}
\usepackage{amsmath,amsfonts,amssymb,amscd,enumerate,verbatim,hyperref}
\usepackage{amsthm}
\usepackage[latin1]{inputenc}
\usepackage{amscd}
\usepackage[capitalise, noabbrev]{cleveref}
\usepackage{multicol}
\usepackage{color}
\usepackage{tikz}
\input xy
\xyoption{all}

\usepackage{fullpage}
\usepackage{caption}

\def\NZQ{\mathbb}               
\def\NN{{\NZQ N}}

\def\GG{{\NZQ G}}

\def\GG{{\mathcal G}}
\def\frk{\mathfrak}               

\def\Phi{{\frk N}}
\def\opn#1#2{\def#1{\operatorname{#2}}} 
\opn\chara{char} 
\opn\length{\ell} 
\opn\pd{pd} 
\opn\rk{rk}
\opn\projdim{proj\,dim} 
\opn\injdim{inj\,dim} 
\opn\rank{rank}
\opn\depth{depth} 
\opn\grade{grade} 
\opn\height{height}
\opn\embdim{emb\,dim} 
\opn\codim{codim}

\opn\Tr{Tr} 
\opn\bigrank{big\,rank}
\opn\superheight{superheight}
\opn\lcm{lcm}
\opn\trdeg{tr\,deg}
\opn\reg{reg} 
\opn\lreg{lreg} 
\opn\ini{in} 
\opn\lpd{lpd}
\opn\size{size}
\opn\mult{mult}
\opn\dist{dist}
\opn\cone{cone}
\opn\lex{lex}
\opn\rev{rev}
\opn\div{div} \opn\Div{Div} \opn\cl{cl} \opn\Cl{Cl}
\opn\Spec{Spec} \opn\Supp{Supp} \opn\supp{supp} \opn\Sing{Sing}
\opn\Ass{Ass} \opn\Min{Min}
\opn\Ann{Ann} \opn\Rad{Rad} \opn\Soc{Soc}
\opn\Syz{Syz} \opn\Im{Im} \opn\Ker{Ker} \opn\Coker{Coker}
\opn\Am{Am} \opn\Hom{Hom} \opn\Tor{Tor} \opn\Ext{Ext}
\opn\End{End} \opn\Aut{Aut} \opn\id{id} \opn\ini{in}

\opn\nat{nat}
\opn\pff{pf}
\opn\Pf{Pf} \opn\GL{GL} \opn\SL{SL} \opn\mod{mod} \opn\ord{ord}
\opn\Gin{Gin}
\opn\Hilb{Hilb}\opn\adeg{adeg}\opn\std{std}\opn\ip{infpt}
\opn\Pol{Pol}
\opn\sat{sat}
\opn\Var{Var}
\opn\Gen{Gen}

\opn\aff{aff} \opn\con{conv} \opn\relint{relint} \opn\st{st}
\opn\lk{lk} \opn\cn{cn} \opn\core{core} \opn\vol{vol}
\opn\link{link} \opn\star{star}
\opn\gr{gr}

\def\Ac{{\mathcal A}}

\def\Mc{{\mathcal M}}
\def\Nc{{\mathcal N}}

\newcommand{\qand}{\quad \mbox{and} \quad}

\def\pot#1#2{#1[\kern-0.28ex[#2]\kern-0.28ex]}

\opn\dirlim{\underrightarrow{\lim}}
\opn\inivlim{\underleftarrow{\lim}}
\def\Implies{\ifmmode\Longrightarrow \else
        \unskip${}\Longrightarrow{}$\ignorespaces\fi}
\def\implies{\ifmmode\Rightarrow \else
        \unskip${}\Rightarrow{}$\ignorespaces\fi}
\def\iff{\ifmmode\Longleftrightarrow \else
        \unskip${}\Longleftrightarrow{}$\ignorespaces\fi}

\let\:=\colon

\newtheorem{Theorem}{Theorem}[section]
\newtheorem{Lemma}[Theorem]{Lemma}
\newtheorem{Corollary}[Theorem]{Corollary}
\newtheorem{Proposition}[Theorem]{Proposition}

\theoremstyle{definition}
\newtheorem{Remark}[Theorem]{Remark}

\newtheorem{Example}[Theorem]{Example}

\newtheorem{Definition}[Theorem]{Definition}

\newtheorem{Conjecture}[Theorem]{Conjecture}

\newtheorem{Notation}[Theorem]{Notation}
\let\epsilon\varepsilon
\let\phi=\varphi
\let\kappa=\varkappa
\opn\dis{dis}
\opn\height{height}
\opn\dist{dist}
\def\pnt{{\raise0.5mm\hbox{\large\bf.}}}

\opn\Lex{Lex}

\begin{document}

\title{Gapfree graphs and powers of edge ideals with linear quotients}

\author[N.~Erey]{Nursel Erey}
\author[S.~Faridi]{Sara Faridi}
\author[T.~H.~H\`a]{T\`ai Huy H\`a}
\author[T.~Hibi]{Takayuki Hibi}
\author[S.~Kara]{Selvi Kara}
\author[S.~Morey]{Susan Morey}

\address{(Nursel Erey) Gebze Technical University, Department of Mathematics, 41400 Gebze, Kocaeli, Turkey}
\email{nurselerey@gtu.edu.tr}
\address{(Sara Faridi) Department of Mathematics \& Statistics, Dalhousie University, 6297 Castine Way, PO BOX 15000, Halifax, NS, Canada B3H 4R2}
\email{faridi@dal.ca}
\address{(T\`ai Huy H\`a) Mathematics Department, Tulane University, 6823 St. Charles Avenue, New Orleans, LA 70118, USA}
\email{tha@tulane.edu}
\address{(Takayuki Hibi) Department of Pure and Applied Mathematics, Graduate School of Information Science and Technology, Osaka University, Suita, Osaka 565--0871, Japan}
\email{hibi@math.sci.osaka-u.ac.jp}
\address{(Selvi Kara) Department of Mathematics, Bryn Mawr College, Bryn Mawr, PA 19010, USA}
\email{skara@brynmawr.edu}
\address{(Susan Morey) Department of Mathematics, Texas State University, 601 University Dr., San Marcos, TX 78666, USA}
\email{morey@txstate.edu}



\begin{abstract}  
Let $I(G)$ be the edge ideal of a gapfree graph $G$.
An open conjecture of Nevo and Peeva states that $I(G)^q$ has linear resolution for $q\gg 0$. We present a promising approach to this challenging conjecture by investigating the stronger property of linear quotients. Specifically, we make the 
conjecture that if $I(G)^q$ has linear quotients for some integer $q\geq 1$, then $I(G)^{s}$ has linear quotients for all $s\geq q$.  We give a partial solution to this conjecture, and identify conditions under which only finitely many powers need to be checked. It is known that if $G$ does not contain a cricket, a diamond, or a $C_4$,  then $I(G)^q$ has linear resolution for $q \geq 2$.  We construct a family of gapfree graphs $G$ containing cricket, diamond, $C_4$ together with $C_5$ as induced subgraphs of $G$ for which $I(G)^q$ has linear quotients for $q \ge 2$.     
\end{abstract}

\maketitle

\section{Introduction}
Let $G$ be a finite \emph{simple} graph, i.e., a graph that does not contain loops or multiple edges, with vertex set $V(G) = \{1, \ldots, n\}$ and edge set $E(G)$.  Let $S = K[x_1, \ldots, x_n]$ denote the polynomial ring in $n$ variables over a field $K$.  The \emph{edge ideal} of $G$, denoted by $I(G)$, is the ideal of $S$ generated by squarefree quadratic monomials $x_ix_j$ for all $\{i, j\} \in E(G)$ (cf. \cite{V}). The study of edge ideals has become a central topic in combinatorial commutative algebra, showcasing the deep interplay between combinatorial methods and algebraic structures; see, for instance, \cite{BN, DN, EO, HNT, LTT, P1, P2, SF1, SF2, V1} for various recent aspects of this study. 
The present paper builds upon this rich foundation, advancing the understanding of powers of edge ideals and their algebraic properties.

A fundamental question in examining homogeneous ideals is when these ideals have linear resolution. In 2004, Herzog, Hibi and Zheng \cite{HHZ} showed that if $I(G)$ has linear resolution then all powers of $I(G)$ also have linear resolution. By a result of Fr\"oberg \cite{F}, this condition is equivalent to the complementary graph of $G$ being chordal.  A finite simple graph is called \emph{chordal} if every cycle of length at least 4 has a \emph{chord}, i.e., an edge connecting two non-consecutive vertices on the cycle. The proof of Herzog, Hibi and Zheng relies on Dirac's theorem on chordal graphs, a cornerstone of classical graph theory, and leverages Gr\"obner bases techniques.  

Furthermore, Francisco, H\`a and Van Tuyl (see \cite{NP}) observed in unpublished work that if $I(G)^q$ has linear resolution for some $q \in \NN$ then $G$ must be gapfree. A finite simple graph $G$ is \emph{gapfree} if, for any two edges $e$ and $e'$ with $e \cap e' = \emptyset$, there is $f \in E(G)$ with $e \cap f \neq \emptyset$ and $e' \cap f \neq \emptyset$. This led to the natural question: 

\begin{center}
\emph{``Is $G$ being gapfree both a necessary and sufficient condition for $I(G)^q$ to have linear resolution for all $q \ge 2$?''} 
\end{center}
In 2013, Nevo and Peeva \cite{NP} answered this question negatively by constructing a gapfree graph $G$ for which $I(G)^2$ does not have linear resolution. In light of this investigation, Nevo and Peeva \cite{NP} raised the following conjecture:

\begin{Conjecture} 
\label{NP}
If $G$ is gapfree, then $I(G)^q$ has linear resolution for all $q \gg 0$.
\end{Conjecture}

Conjecture \ref{NP} has garnered significant interest (cf. \cite{BBH, Big, E18, E19, MV} and references therein), but remains unresolved in general. While it has been verified for $q = 2,3$ when $I(G)$ has small regularity (see \cite{MV}) and for \emph{random} graphs with restricted edge probability (see \cite{BY}), progress has been limited. 
A promising approach to this challenging conjecture is to explore the stronger property of having \emph{linear quotients}. This property has also been widely studied and, in particular, if a monomial ideal generated in a single degree has linear quotients then it has linear resolution (cf. \cite{HHgtm260}). Thus, as observed before, if $I(G)^q$ has linear quotients, for some $q \in \NN$, then $G$ is necessarily a gapfree graph. We make the following conjecture:

\begin{Conjecture} 
\label{LinearQuotient}
If $I(G)^q$ has linear quotients for some $q \in \NN$, then $I(G)^{s}$ has linear quotients for all $s \ge q$.
\end{Conjecture}

For a subset $W \subset V(G)$, the \emph{induced subgraph} $G_W$ is the graph with vertex set $W$ consisting of all edges $\{i,j\} \in E(G)$ where $i,j \in W$. It was shown in \cite {B, E18, E19} that $I(G)^q$ have linear resolution for all $q \geq 2$ if $G$ is gapfree and contains no induced subgraph isomorphic to: (i) cricket; (ii) diamond; (iii) $C_4$ (see Figure \ref{fig:subgraphs}).  

\begin{figure}[h]
\centering
\begin{tikzpicture}[scale=1.0]
\coordinate (a) at (-1,1){};
\coordinate (b) at (1,1){};
\coordinate (c) at (0,0){};
\coordinate (d) at (-1,-1){};
\coordinate (e) at (1,-1){};
\fill(a)circle(0.7mm);
\fill(b)circle(0.7mm);
\fill(c)circle(0.7mm);
\fill(d)circle(0.7mm);
\fill(e)circle(0.7mm);
\draw(a)--(c)--(e);
\draw(b)--(c)--(d);
\draw(d)--(e);
\end{tikzpicture}
\, \, \, \, \, \, \, \, \, \, 
\begin{tikzpicture}[scale=1.0]
\coordinate (a) at (0,1){};
\coordinate (b) at (-1,0){};
\coordinate (c) at (1,0){};
\coordinate (d) at (0,-1){};
\fill(a)circle(0.7mm);
\fill(b)circle(0.7mm);
\fill(c)circle(0.7mm);
\fill(d)circle(0.7mm);
\draw(a)--(b)--(d)--(c)--cycle;
\draw(b)--(c);
\end{tikzpicture}
\, \, \, \, \, \, \, \, \, \, 
\begin{tikzpicture}[scale=1.0]
\coordinate (a) at (0,1){};
\coordinate (b) at (-1,0){};
\coordinate (c) at (1,0){};
\coordinate (d) at (0,-1){};
\fill(a)circle(0.7mm);
\fill(b)circle(0.7mm);
\fill(c)circle(0.7mm);
\fill(d)circle(0.7mm);
\draw(a)--(b)--(d)--(c)--cycle;
\end{tikzpicture}
\caption{Cricket; Diamond; $C_4$} \label{fig:subgraphs}
\end{figure}
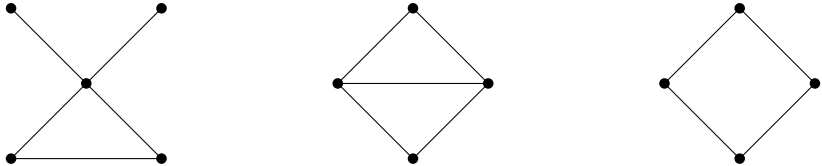 

For instance, the pentagon $C_5$ is gapfree, but its complementary graph is not chordal. Consequently, $I(C_5)$ does not have linear resolution. However, $C_5$ does not contain any induced subgraph that is isomorphic to a cricket (or a diamond, or $C_4$), so $I(C_5)^q$ has linear resolution for all $q \ge 2$. On the other hand,  it was shown (cf. \cite[Corollary 10.1.7 and Theorems 10.1.9 and 10.2.5]{HHgtm260}) that if the complementary graph of $G$ is chordal then all powers of $I(G)$ have linear quotients. Beside the anticycles, as exhibited in a recent preprint \cite{B+} which appears after the first draft of this paper was written, there are currently no other known examples of a gapfree graph $G$ whose complementary graph is not chordal, yet for which all powers $I(G)^q$ with $q \geq 2$ have linear quotients. This motivates the search for natural classes of such graphs. 
 
In this paper, we establish a framework for addressing Conjecture \ref{LinearQuotient}.  In Section \ref{sec.squarefree monomial duplication}, we examine the duplication technique, which when applied to graphs uses duplication of vertices in $G$ to produce new graphs $G^x$. This construction is standard in graph theory and has been used to obtain interesting algebraic results (cf. \cite{FHVT}). We show that for a general squarefree monomial ideal $I$ if a power $I^q$ has linear quotients, then so does $(I^x)^q$ (Proposition \ref{p:duplication}).  In Section~\ref{sec.dup}, we show similar result holds for graphs when vertex duplication is replaced by graph expansion, where a vertex is replaced by a clique, as long as the resulting graph is still gapfree (Theorem \ref{t:complete}).

Section \ref{sec.PureOrder} studies Conjecture \ref{NewOrleans}, which presents a strong case for Conjecture \ref{LinearQuotient} by explicitly describing a possible linear quotient ordering for $I(G)^s$, for $s \ge q$, from that of $I(G)^q$. Particularly, Proposition \ref{pentagon} shows that Conjecture \ref{NewOrleans} holds for the pentagon $C_5$. In other words, all powers $I(C_5)^q$ with $q \geq 2$ have linear quotients.
Section \ref{sec.Pentagon} extends the study in Section \ref{sec.PureOrder} to a special class of graphs which are gapfree but may contain critical induced subgraphs such as crickets, diamonds and $C_4$. In Propositions \ref{p:SusanSelvi} and \ref{l:pentagon_4inner_edges}, it is shown that, for these graphs, all powers $I(G)^q$ with $q \geq 2$ have linear quotients. The main result of the section (Theorem~\ref{pasadenaAIM}) extends the aforementioned result to an infinite family of graphs.

Finally, Section \ref{sec.Partial} presents a refined analysis of modified versions of Conjectures \ref{NewOrleans} and \ref{LinearQuotient}, culminating in Theorem \ref{thm.LinQuotients}, which exhibits that if $I(G)^2, \ldots, I(G)^7$ have linear quotients with a special ordering of their generators, then $I(G)^q$, for all $q \geq 2$, have linear quotients. This result reduces the verification of Conjectures \ref{NewOrleans} and \ref{LinearQuotient} to finitely many powers and, therefore, suggests that resolving these conjectures may be within reach.


\section{Duplicating a variable and linear quotients of powers}\label{sec.squarefree monomial duplication}

Let $I$ be an ideal of $S$ which is generated by monomials of the same degree and let $\{u_1, \ldots, u_r\}$ be the unique minimal monomial generators of $I$.  We then call $u_1, \ldots, u_r$ \emph{the} generators of $I$ and use the notation $\mathcal{G}(I)$ to denote the set $\{u_1,\ldots, u_r\}$. We say that $I$ has \emph{linear quotients} if there is an ordering $u_1  > \cdots > u_r$ of the generators of $I$ for which the colon ideal $(u_1, \ldots, u_{i}):u_{i+1}$ is generated by a subset of $\{x_1, \ldots, x_n\}$ for each $1 \leq i < r$.  If $I$ has linear quotients for  $u_1  > \cdots > u_r$, we call this ordering a \emph{linear quotients order} for $I$. If $I$ has linear quotients, then $I$ has linear resolution (\cite[Proposition 8.2.1]{HHgtm260}). 

When showing an ideal has a linear quotients order, we will frequently examine individual monomial pairs. For simplicity, we will use the following notation below for two monomials $M$ and $N$
$$M:N=\frac{M}{\gcd(M,N)}.$$
In particular, $M \mid (M:N) N$ and $M:N$ is the minimal monomial under inclusion for which this property holds.

The following statement is an equivalent version of the definition of an ideal with linear quotients. The proof follows immediately from the definitions and so is not included.

\begin{Lemma}[{\bf Linear quotients}]\label{l:lq-criterion} Let $I$ be a monomial ideal, and fix an order $u_1>u_2>\cdots>u_r$ on the
minimal generating set of $I$. Then the following statements are equivalent.
\begin{enumerate}
    \item  $I$ has linear quotients with respect to the order above;
\item if $\deg(u_\alpha : u_\beta)>1$ for some indices $\alpha < \beta$, then there is  some index $\gamma < \beta$ 
such that  $u_\gamma : u_\beta=z$ where $z$ is a variable and $z \mid (u_\alpha : u_\beta)$.
\end{enumerate}
\end{Lemma}

The equivalence of items $(1)$ and $(2)$ will be used throughout the paper when showing an ideal has linear quotients. The first result we show is that linear quotients are preserved under variable duplication. Since there are multiple ways to duplicate a variable, we establish a definition and notation.

\begin{Definition}[{\bf Duplicating a variable}]\label{d:duplication-new}
Let $I$ be a squarefree monomial ideal, suppose $x$ is a variable, and suppose $m_1,\ldots,m_q$ are all the monomials in $\GG(I)$  that are divisible by $x$. Let $y$ be a new variable that does not divide any generator of $I$. The \emph{duplicate ideal of $I$ by $x$} is the ideal
$$I^x=I+\big (\frac{m_1}{x}\cdot y, \ldots, \frac{m_q}{x}\cdot y \big ).$$
If $x$ does not divide any generator of $I$, then $I^x=I$.
\end{Definition}

We next show that linear quotients of squarefree monomial ideals are preserved under variable duplication.  For a monomial $M$ and a variable $x$, we use the notation $\deg_x(M)$  to denote the maximum value of $i$ such that $x^i \mid M$.

\begin{Proposition}\label{p:duplication} Let $I$ be an equigenerated  squarefree monomial ideal,  $x$ a variable, and $s$ a positivie integer. If $I^s$ has linear quotients, then $(I^x)^s$ has linear quotients.
\end{Proposition}

\begin{proof}
Fix a variable  $x$. If $x$ does not appear in $\GG(I)$ then $I^x=I$ and we are done by assumption, so we assume $x$ divides a generator of $I$. Let  $y$ be the duplicate of $x$ in $I^x$. Note that, for each $u\in \GG(I^s)$ with $\deg_x(u)=d>0$, the following $d$ monomials obtained from $u$ 
\begin{equation}\label{eq:new generator}
  \frac{\,u\,}{x}\cdot y,
  \frac{\,u\,}{x^2}\cdot y^2, \dots ,
  \frac{\,u\,}{x^d}\cdot y^d 
\end{equation}
are  the elements of $\GG((I^x)^s)$ which are not  in $\GG(I^s)$. 
 Order the generators of $(I^x)^s$ by extending a linear quotients order  of $I^s$  in the following way: If $u\in \GG(I^s)$ with $\deg_x(u)>0$,  insert the new generators obtained from $u$ immediately following $u$ in the order they appear in (\ref{eq:new generator}).

Let $w_1>\cdots > w_q$ denote the generators of $(I^x)^s$ in this order.  We claim that this  is a linear quotients order of $(I^x)^s$.  
Let $t\geq 2$ and $$J:=(w_1,\dots ,w_{t-1}):w_t.$$ Consider $w_{\ell}:w_t$ for $\ell<t$.

\subsection*{CASE 1:} Suppose $w_t \in \GG (I^s)$, i.e.,  $w_t$ is not divisible by $y$. 

$\cdot$ If $w_{\ell} \in  \GG (I^s)$, we are done since $I^s$ has linear quotients. 

$\cdot$ If $w_{\ell} \notin \GG (I^s)$, then $w_{\ell}$ is divisible by $y$, which implies that $w_{\ell}:w_t\in (y)$. 

\begin{itemize}
\item[-] If $x \in \GG(J)$, there exists $w_j \in \GG (I^s)$ with $j<t$ such that $w_j:w_t=x$. Let $w_k= w_j y/x$. Notice that $j<k< t$ by the definition of the order on $\GG ((I^x)^s)$. Thus, we have $w_k:w_t=y \in \GG(J)$, as desired. 

\item[-] If $x \notin \GG(J)$, then let  $\deg_x (w_{\ell})=\alpha_x\geq 0$ and $ \deg_y (w_{\ell})=\alpha_y\geq 1$. Then $w_{\ell}=x^{\alpha_x} y^{\alpha_y} M$  for a monomial $M$. Consider the monomial  $w_k=x^{\alpha_x+\alpha_y}M \in \GG(I^s)$ and notice that  $k<{\ell}$ by the definition of the order. Since $I^s$ has linear quotients,  there exists $w_j \in \GG (I^s)$ with $j<t$ and a variable $z$ such that $z=(w_j:w_t)$ and $z \mid (w_k:w_t)$. Note that  $z\neq x$ since $x \notin \GG(J)$, and $z\neq y$ since $w_j,w_t \in \GG (I^s)$. 
Therefore, $w_{\ell}:w_t$ is divisible by $z$, and hence $J$ is linear. 
\end{itemize}

\subsection*{CASE 2:} Suppose $w_t \notin \GG (I^s)$, i.e.,  $w_t$ is divisible by $y$.

Our first goal is to show  $x \in \GG(J)$. Set $w_j= w_tx/y$. Then $j<t$ by the definition of the order. Thus, $w_j:w_t=x \in \GG(J)$.

Let  $w_{\ell}=x^{\alpha_x} y^{\alpha_y} M$ for a monomial $M$ where $\alpha_x=\deg_x (w_{\ell})$ and $\alpha_y=\deg_y(w_{\ell})$. Similarly, let $w_t=x^{\beta_x} y^{\beta_y} N$ for a monomial $N$ where 
$\beta_x=\deg_x (w_t)$ and $\beta_y=\deg_y(w_t)\geq 1$. 
 Consider the  monomials $w_L= x^{\alpha_x+\alpha_y} M \text{ and } w_T=x^{\beta_x+\beta_y} N$ in $\GG(I^s)$  
where $L \le {\ell}$ and $T<t$ by the definition of the order.

If $\alpha_x>\beta_x$, then  $x\in \GG(J)$ divides $w_{\ell}:w_t$, and we are done. Suppose $\alpha_x\leq \beta_x$. This implies that ${\ell}<T$  by the definition of the order. Since $I^s$ has linear quotients and $L<T$,  there exists $w_j \in \GG(I^s)$ with $j<T$ such that $w_j:w_T=z$ for some variable $z$ dividing $w_L:w_T$. Note that $z \ne y$. Let $\alpha=\alpha_x+\alpha_y$ and $\beta=\beta_x+\beta_y$.

\textit{Case 2.1:} Suppose  $\alpha \leq \beta$. In this case,     $w_L:w_T=M:N$ which implies that $w_{\ell} :w_t$ is divisible by $z$. It remains to show $z \in \GG(J)$.  Since $w_j:w_T=z$, the monomial $w_j$ has one of the following forms: 
\begin{enumerate}
    \item[(a)]  $w_j= x^{\beta-1} z N$ or
    \item[(b)] $w_j= x^{\beta} N z/z'$ for a variable $z'$ dividing $N$ where $z'\neq z$.
\end{enumerate}

If (a) holds,  let $w_{k} = x^{\beta_x} y^{\beta_y-1} zN$ where $j\leq k <T<t$ by the definition of the order. Thus, $w_{k}:w_t=z\in \GG(J)$, as desired. Similarly, if (b) holds, let $w_{k}=x^{\beta_x} y^{\beta_y} Nz/z'$  where $j< k<T<t$ by the definition of the order. Hence, $w_{k}:w_t=z\in \GG(J)$.

\textit{Case 2.2:} Suppose $\alpha >  \beta$.  It then follows from the assumption $\alpha_x\leq \beta_x$ that $\alpha_y>\beta_y$.  In this case,    $w_{\ell}:w_t= y^{\alpha_y-\beta_y} (M:N)$ and  $w_L:w_T=x^{\alpha -\beta} (M:N)$. So, $w_{\ell}:w_t$  is divisible by both $y$ and $z$. It remains to show either $y \in \GG(J)$ or $z \in \GG(J)$.  Notice that it is possible to have $z=x$. 

Suppose $z=x$. Then $w_j= x^{\beta+1} N/z'$ for some variable $z'$ dividing $N$ since $w_j:w_T=z$.  Let $w_k= x^{\beta_x}y^{\beta_y+1} N/z'$. Note that   $j<k<T<t$ by the definition of the order and $w_k:w_t= y\in \GG(J)$.

Suppose $z \neq x$. This implies that $z$ divides $M:N$. In this case, $w_j$ is in one of the forms described in (a) or (b) from Case 2.1. The remainder of the proof follows similarly and it results with $z\in \GG(J)$.
\end{proof}

\section{Expansion, and linear quotients of powers of edge ideals}\label{sec.dup}

For the remainder of the paper, we focus our attention on edge ideals of graphs.
Let $G$ be a finite graph on $V(G) = \{1, \ldots, n\}$ with $E(G) = \{e_1, \ldots, e_s\}$, and let $S = K[x_1, \ldots, x_n]$ be the polynomial ring in $n$ variables over a field $K$.  For an edge $e = \{i,j\} \in E(G)$, we define $u_e = x_i x_j \in S$.  Recall that the edge ideal of $G$ is the ideal $I(G)$ of $S$ which is generated by those monomials $u_e$ with $e \in E(G)$.  To simplify the notation, unless misleading, we write $e$ instead of $u_e$.  For example, $e_1^3, e_1e_2^2 \in I(G)^3$. 
In addition, we may write $ij\in E(G)$ instead of $\{i,j\} \in E(G)$. Lastly, we may refer to and use the variables $x_1,\ldots, x_n$ as vertices of $G$.

\begin{Definition}\label{d:duplication-expansion}
Let $G$ be a finite graph and fix a vertex $x$ of $G$.  
\begin{itemize}

\item The graph $G^x$, called the \emph{duplication} of $G$ at $x$, is a finite graph obtained by adding a new vertex $y$ to $G$ with $$E(G^x) = E(G) \cup \{yb : xb \in E(G)\}.$$  
We say that $G^x$ is obtained from $G$ by \emph{duplicating} the vertex $x$ of $G$.  
Note that $xy \not\in E(G^x)$, and also
$I(G^x)=I(G)^x$ in the sense of \cref{d:duplication-new}. 
We call the new vertex $y$ the \emph{duplicate} of $x$ in $G^x$.

\item The graph $G^{[x]}$, called the \emph{expansion} of $G$ at $x$, is the graph with vertex set $V(G^{[x]})=V(G^x)$ and $E(G^{[x]})=E(G^x)\cup \{xy\}$.

\item  For any $W\subseteq V(G^{[x]})=V$ and any monomial $m\in I(G^{[x]})^s$, define $m_W$ to be the monomial formed by localizing $m$ at $W$. That is, if $m=\prod_{v_i\in V}v_i^{s_i}$, then $m_W = \prod_{v_i\in W}v_i^{s_i}$. 
\end{itemize}
\end{Definition}

The following example illustrates the definitions above.

\begin{Example}
Let $G$ be the graph $C_4$ labeled as in the left hand graph below. Then $G^x$ is the central graph and $G^{[x]}$ is the graph on the right. Consider the monomial $m=(xy)^2(xa)(bc) \in I(G^{[s]})^4$. If $W=\{x,c\}$, then $m_W =x^3c.$

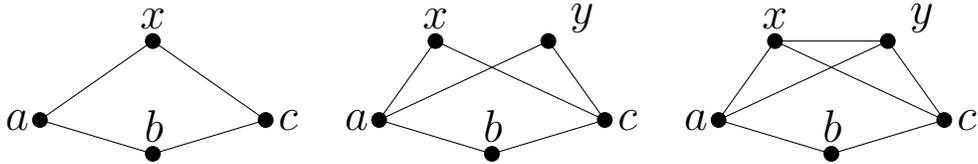
\begin{figure}[h]
\centering
\begin{tabular}{ccc}
\begin{tikzpicture}[scale=1.5]
\coordinate (x) at (0,1){};
\coordinate (a) at (-1,0.3){};
\coordinate (b) at (0,0){};
\coordinate (c) at (1,0.3){};

\fill(x)circle(0.7mm);
\fill(a)circle(0.7mm);
\fill(b)circle(0.7mm);
\fill(c)circle(0.7mm);

\draw(x)--(a)--(b)--(c)--cycle;

\draw(0,1.2)node{{\Large $x$}};
\draw(-1.2,0.3)node{{\Large $a$}};
\draw(1.2,0.3)node{{\Large $c$}};
\draw(0.01,0.25)node{{\Large $b$}};
\end{tikzpicture}

&

\begin{tikzpicture}[scale=1.5]
\coordinate (x) at (-0.5,1){};
\coordinate (y) at (0.5,1){};
\coordinate (a) at (-1,0.3){};
\coordinate (b) at (0,0){};
\coordinate (c) at (1,0.3){};

\fill(x)circle(0.7mm);
\fill(y)circle(0.7mm);
\fill(a)circle(0.7mm);
\fill(b)circle(0.7mm);
\fill(c)circle(0.7mm);

\draw(x)--(a)--(b)--(c)--cycle;
\draw(y)--(a);
\draw(y)--(c);

\draw(-0.5,1.2)node{{\Large $x$}};
\draw(0.8,1.2)node{{\Large $y$}};
\draw(-1.2,0.3)node{{\Large $a$}};
\draw(1.2,0.3)node{{\Large $c$}};
\draw(0.01,0.25)node{{\Large $b$}};
\end{tikzpicture}

&

\begin{tikzpicture}[scale=1.5]
\coordinate (x) at (-0.5,1){};
\coordinate (y) at (0.5,1){};
\coordinate (a) at (-1,0.3){};
\coordinate (b) at (0,0){};
\coordinate (c) at (1,0.3){};

\fill(x)circle(0.7mm);
\fill(y)circle(0.7mm);
\fill(a)circle(0.7mm);
\fill(b)circle(0.7mm);
\fill(c)circle(0.7mm);

\draw(x)--(a)--(b)--(c)--cycle;
\draw(y)--(a);
\draw(y)--(c);
\draw(x)--(y);

\draw(-0.5,1.2)node{{\Large $x$}};
\draw(0.8,1.2)node{{\Large $y$}};
\draw(-1.2,0.3)node{{\Large $a$}};
\draw(1.2,0.3)node{{\Large $c$}};
\draw(0.01,0.25)node{{\Large $b$}};
\end{tikzpicture}

\end{tabular}
\caption{Graph $G$, $G^x$ and $G^{[x]}$}
\end{figure}
\end{Example}

We state some useful facts regarding the duplicated vertex $x$ and its duplication $y$. The proofs are straightforward from the definitions and so are not included.  First, recall that 
$$N_G(x)=\{ z \in V(G) : xz \in E(G)\} 
\qand
N_G [x]= N_G(x) \cup \{x\}.$$

\begin{Lemma}\label{l:basics}
    Let $G$ be a gapfree graph. Using the notation of \Cref{d:duplication-expansion},
    \begin{itemize}
        \item $N_{G^x}(x) = N_{G^x}(y)=N_G(x) = N_{G^{[x]}}\setminus \{y\}$.
        \item $N_{G^{[x]}}[x]=N_{G^{[x]}}[y]$.
        \item $G^{[x]}$ is gapfree if and only if $V(G) \setminus N_G[x]$ is an independent set.
    \end{itemize}
\end{Lemma}

We are now ready to prove the main theorem of this section.

\begin{Theorem}\label{t:complete}
    Let $G$ be a graph such that $I(G)^s$ has linear quotients for some $s \geq 1$. Fix $x\in V(G)$. Then $I(G^{[x]})^s$ has linear quotients if and only if $G^{[x]}$ is gapfree.
\end{Theorem}

\begin{proof}
First note that if $I(G^{[x]})^s$ has linear quotients for some $s\geq 1$ then $G^{[x]}$ is gapfree, thus we focus on the converse.

For any monomial generator $w$ of $I(G^{[x]})^s$ 
define
\[\mu(w)=\min \left\{i:  \frac{w}{(xy)^i} \in \mathcal{G}(I(G^x)^{s-i})\right\}.\]
Note that $\mu(w)=0$ if and only if $w$ is a minimal generator of $I(G^x)^s$. Set $V=V(G^{[x]})$. Partition $V$ into sets:
\begin{eqnarray*}
    Z &=& \{x,y\} \\
    A &=& N_G(x) \\
    B &=& V(G^{[x]})\setminus N_{G^{[x]}}[x]
\end{eqnarray*}
By \Cref{l:basics}, $B$ is an independent set. Fix an order $>_B$ on the elements of $B$.  By Proposition \ref{p:duplication}, there is a linear quotients order $w_1,w_2,\dots ,w_q$ for $I(G^x)^s$. We  extend this order to a linear quotients order $w_1,\dots ,w_q,w_{q+1},\dots ,w_r$ for $I(G^{[x]})^s$. To simplify the notation, we adopt the convention that $w_1>w_2>\dots >w_r$. 
For generators $w$ and $w'$ of $I(G^{[x]})^s$ not belonging to $I(G^x)^s$,
set $w>w'$ if:
\begin{enumerate}
    \item $\mu(w) < \mu(w')$, or
    \item $\mu(w)=\mu(w')$ and $\deg(w_Z)<\deg(w'_Z)$, or
    \item $\mu(w)=\mu(w')$, $\deg(w_Z)=\deg(w'_Z)$, and $|\deg_x(w)-\deg_y(w)|<|\deg_x(w')-\deg_y(w')|$, or
    \item $\mu(w)=\mu(w')$, $\deg(w_Z)=\deg(w'_Z)$, $|\deg_x(w)-\deg_y(w)|=|\deg_x(w')-\deg_y(w')|$, and $w_{B}>w'_{B}$ in the lexicographic ordering on $B$, using the order $>_B$ on the elements of $B$ fixed above.
\end{enumerate}

Let $w=w_i$ for some $i\geq 2$. We  show that $J=(w_1,\dots ,w_{i-1}):w_i$ is generated by variables. We may assume that $\mu(w)\neq 0$ by the definition of the order and Proposition \ref{p:duplication}. 
Since $B$ is an independent set, the edges of $G^{[x]}$ can be partitioned as follows:
\begin{eqnarray*}
    E_1 &=& \{xy\}\\
    E_2 &=& \{ab \mid a \in A, b\in B\}\\
    E_3 &=& \{az \mid a \in A, z\in Z\}\\
    E_4 &=& \{aa' \mid a, a' \in A\}.
\end{eqnarray*}
We make two observations which will be used in the proofs.
\begin{itemize}
    \item[(i)]\label{rule 1} Let $w$ be a generator of $I(G^{[x]})^s$
    with $\mu(w)\neq 0$. For any representation $$w=(xy)^{\mu(w)}e_1\dots e_{s-\mu(w)}$$ where  $e_1,\dots ,e_{s-\mu(w)}$ are in $E(G^x)$, we have $e_j\notin E_4$ for all $j=1,\dots ,s-\mu(w)$.
    \item[(ii)]\label{rule 2} $A$ is a subset of the set of generators of $J$.   
\end{itemize}
\textit{Proof of (i).} Assume to the contrary that $e_1=a_1a_2$ for some $a_1,a_2\in A$. Then we can write
\[w=(xy)^{\mu(w)-1}(xa_1)(ya_2)e_2\dots e_{s-\mu(w)}\]
which is a contradiction to the minimality of $\mu(w)$. 

\textit{Proof of (ii).} Let $a\in A$. Then $u:=\frac{w}{xy}(xa)$ is a minimal generator of $I(G^{[x]})^s$ with $\mu(u)\leq \mu(w)-1$. Therefore $u>w$ and $u:w=a\in J$, proving the claim.

\medskip

Let $v>w$. We show that there is a variable generator of $J$ dividing $v:w$.


\noindent\textbf{CASE 1:} Suppose that $\mu(v)<\mu(w)$. Using observation (i), each $e_i$ in the representation of $w$ either has precisely one endpoint in $A$ or is in $E_1$. It follows that
\[\deg(v_A) \geq s-\mu(v)>s-\mu(w)=\deg(w_A).\]
Therefore, there is an $a\in A$ such that $a$ divides $v_{A}:w_{A}$. Since $v_{A}:w_{A}$ divides $v:w$, the proof follows from observation (ii).


\noindent\textbf{CASE 2:}  Suppose that $1\leq \mu(v)=\mu(w)$. As above, from observation (i) it follows that 
\[\deg(v_{A})=s-\mu(v)=s-\mu(w)=\deg(w_{A}).\]

Now by observation (ii), we may assume for every $a\in A$ we have $\deg_a(v) \le \deg_a(w)$. However, since
$$\deg(v_A) = \sum_{a\in A} \deg_a(v) \leq \sum_{a\in A} \deg_a(w) = \deg(w_A)$$
the inequality cannot be strict for any $a\in A$. Thus
for the remainder of the proof, we assume that $v_{A}=w_{A}$.

\textit{Case 2.1:} Suppose that $\deg(v_Z)<\deg(w_{Z})$. Let $\gamma=\mu(v)=\mu(w)$. By observation (i) we may represent $v=(xy)^{\gamma}e_1\dots e_m f_1\dots f_k$ and $w=(xy)^{\gamma}g_1\dots g_{m'} h_1\dots h_{k'}$ with $e_j, g_j\in E_3$ and $f_j, h_j\in E_2$ for all $j$. Note that 
\[2\gamma+m=\deg(v_Z)<\deg(w_Z)=2\gamma+m'.\]

On the other hand, since $m+k=m'+k'$, it follows that $k>k'$. Therefore
\[\deg((f_1\dots f_k)_A)=k>k'=\deg((h_1\dots h_{k'})_A).\]
Then since $v_A=w_A$, there exists $a \mid (f_1\dots f_k)_A$ such that $a \mid (g_1\dots g_{m'})_A$. Without loss of generality, we may assume that $f_1=ab$ for some $b\in B$ and $g_1=az$ for some $z\in Z$. If $b\mid (v:w)$, set $w':=\frac{w}{g_1}f_1 = \frac{w}{z}b$.  Then $w'>w$ with $w':w=b$, which completes the proof. 

Assume that $b\nmid (v:w)$. Then, without loss of generality, we may assume that $h_1=a'b$ for some $a'\in A$. Then we can re-write $w$ as follows:
\[w=(xy)^{\gamma}(a'z)g_2\dots g_{m'}(ab)h_2\dots h_{k'}.\]
Then, the same argument can be applied recursively to $v/(ab)$ and $w/(ab)$.

\textit{Case 2.2:} Suppose that $\deg(v_Z)=\deg(w_Z)$. 

\textit{Case 2.2-(i):} Suppose that $|\deg_x(v)-\deg_y(v)|<|\deg_x(w)-\deg_y(w)|$. Since the right side of the inequality cannot be $0$, 
without loss of generality, assume that $\deg_x(w)>\deg_y(w)$. Then since $\deg(v_Z)=\deg(w_Z)$, we have $\deg_y(v)>\deg_y(w)$ and $y $ divides $ v:w$. 

Since $\deg_x(w)>\deg_y(w)$ there is an $a\in A$ with $e_1=ax$ in the representation of $w$. By the definition of $G^{[x]}$, $a \in N_{G^{[x]}}(y)$ also. Set $w':= \frac{w}{ax}ay$.  If $|\deg_x(w')-\deg_y(w')|=|\deg_x(w)-\deg_y(w)|$, then $|\deg_x(w)-\deg_y(w)|=1$ and $|\deg_x(v)-\deg_y(v)| = 0$. In this case, $\deg(w_Z)$ is odd and $\deg(v_Z)$ is even, which contradicts $\deg(v_Z)=\deg(w_Z)$.
Thus $|\deg_x(w')-\deg_y(w')|<|\deg_x(w)-\deg_y(w)|$, so $w'>w$ 
and $w':w=y$.

\textit{Case 2.2-(ii):} Suppose that $|\deg_x(v)-\deg_y(v)|=|\deg_x(w)-\deg_y(w)|$. Assume first that $v_Z\neq w_Z$. Then without loss of generality, we may write
\begin{itemize}
    \item[] $\deg_x(v)=r_1=\deg_y(w)$
    \item[] $\deg_x(w)=r_2=\deg_y(v)$
\end{itemize}
for some integers $r_1>r_2$.  Then there exists an $e=ay\in E_3$ in the representation of $w$. Set $w':=\frac{w}{ay}ax=\frac{w}{y}x$. Then $w' >w$ and $x=w':w$ divides $v:w$. Therefore, we may assume $v_Z=w_Z$. Recall that the following statements are now assumed:
\begin{itemize}
    \item $1\leq \mu(v)=\mu(w)=\gamma$,
    \item $v_A=w_A$,
    \item $v_Z=w_Z$.
\end{itemize}


Fix presentations $v=(xy)^{\gamma}f_1f_2\cdots f_t$ and $w=(xy)^{\gamma}g_1g_2\cdots g_t$. For simplicity, we assume that $f_{i_1}\dots f_{i_m}\neq g_{j_1}\dots g_{j_m}$ for all $m$-fold products with $1\leq m\leq t$ since the argument below remains valid recursively after removing the common factors.

Fix $p$ be the greatest in the ordering fixed on $B$ such that $p \mid (v:w).$ We claim that there is $u>w$ such that $u:w=p$. After rearranging the indices, we may write $f_1=a_1p$ for some $a_1\in A$.

\textit{Step 1:} If there exists $\delta_1\in Z$ or $\delta_1\in B$ with $p>_B \delta_1$ such that $g_1=a_1\delta_1$ after rearrangement of indices, set 
\[u=\frac{w}{g_1}f_1=\frac{w}{(a_1\delta_1)}(a_1p)\]
Then $u>w$ in the ordering since $\deg(u_Z)< \deg(w_Z)$ in the first case and $u_B>w_B$ in the latter one, and $u:w=p$. Therefore, we may assume that $g_1=a_1\delta_1$ for some $\delta_1 \in B$ with $\delta_1>_B p$.

\textit{Step 2:} By the choice of $p$, we may set $f_2=a_2\delta_1$ for some $a_2\in A$. As before, if, up to reordering, $g_2=a_2\delta_2$ for some $\delta_2 \in Z$ or $\delta_2 \in B$ with $p>_B \delta_2$ we form the monomial
\[u=\frac{w}{g_1g_2}f_1f_2=\frac{w}{(a_1\delta_1)(a_2\delta_2)}(a_1p)(a_2\delta_1).\]
which satisfies the desired property. Therefore, we may assume that $\delta_2 \in B$ and $\delta_2>_B p$.

\textit{Step 3:} By the choice of $p$, we may set $f_3=a_3\delta_2$ for some $a_3\in A$. As before, if, up to reordering, $g_3=a_3\delta_3$ for some $\delta_3 \in Z$ or $\delta_3 \in B$ with $p>_B \delta_3$ we form the monomial
\[u=\frac{w}{g_1g_2g_3}f_1f_2f_3=\frac{w}{(a_1\delta_1)(a_2\delta_2)(a_3\delta_3)}(a_1p)(a_2\delta_1)(a_3\delta_2).\]
which satisfies the desired property. Therefore, we may assume $\delta_3 \in B$ and $\delta_3>_B p$. The process must stop after $k\le t$ steps when there is no $\delta_k\in B$ left with $\delta_k>_B p$. When the process stops, we have $u >w$ and $u:w=p$ as desired.
\end{proof}

\begin{Corollary}
    Let $G$ be a graph and assume $I(G)^s$ has linear quotients for some $s$. Let $x$ be a vertex of $G$ such that $V (G) \setminus N_G[x]$ is an independent set. Let $H$ be the graph obtained from $G$ by replacing $x$ with a clique. Then $(I(H))^s$ has linear quotients.
 \end{Corollary}

\begin{Remark}
  If $G$ is a (gapfree) graph and $x,y$ are vertices such that $N_G(x)\setminus \{y\} = N_G(y) \setminus \{x\}$, then \Cref{p:duplication} or \Cref{t:complete} can be used to simplify $G$ by identifying $x$ and $y$ prior to determining whether $I(G)^s$ has linear quotients for some $s$, potentially significantly reducing the number of generators that need to be ordered.
\end{Remark}

In the next two sections, we provide examples of gapfree graphs for which powers of their edge ideals have linear quotients starting from the second power. Applying the main results of this section to these graphs yields many additional examples of graphs whose powers have linear quotients.


\section{Ordering of pure powers} \label{sec.PureOrder}

Let $G$ be a graph with $E(G) = \{e_1, \ldots, e_s\}$ and suppose that a power $I(G)^q$ of the edge ideal $I(G)$ has linear quotients (which particularly implies that $G$ is gapfree) with an ordering
\begin{eqnarray*}
\Nc^{(q)} : u_1  > \cdots > u_r
\end{eqnarray*}
of the generators of $I(G)^q$, where pure powers appear in the ordering with 
\[
e_1^q > \cdots >e_s^q. 
\]
Then the ordering
\begin{eqnarray*}
\Nc^{(q+1)} : u_1 e_1 > \cdots > u_r e_1 > u_1 e_2 > \cdots > u_r e_2 > \cdots > u_r e_s    
\end{eqnarray*}
of the generators of $I(G)^{q+1}$, where $u_i e_j$ is omitted if $u_i e_j = u_k e_\ell$ with $\ell < j$, is called the \emph{efficient ordering} constructed from the ordering $\Nc^{(q)}$.  Recursively, we define the efficient ordering $\Nc^{(s)}$, constructed from the ordering $\Nc^{(q)}$, of the generators of $I(G)^{s}$ for $s \geq q$.

 The following is a refinement of Conjecture \ref{LinearQuotient} by explicitly describing the potential linear quotient ordering for higher powers of $I(G)$ coming from that of a given power. It provides a strong means of approaching Conjecture \ref{LinearQuotient}

\begin{Conjecture}
    \label{NewOrleans}
    Let $G$ be a graph and suppose that $I(G)^q$, for some $q \in \NN$, has linear quotients with an ordering $\Nc^{(q)}: u_1 > \dots > u_r$ of its generators. Then, $I(G)^s$ has linear quotients with the efficient ordering $\Nc^{(s)}$, constructed from $\Nc^{(q)}$, of its generators for all $s \ge q$.
\end{Conjecture}

Let $C_5$ be the pentagon on the vertex set $\{a,b,c,d,e\}$ with the edges 
$$e_1=ab, \, \, e_2=bc, \, \, e_3=cd, \, \, e_4=de, \, \, e_5=ea.$$  
One can easily prove directly that $I(C_5)^2$ has linear quotient with the following ordering, with omitting $>$, of its $15$ generators:
\[
a^2b^2, ab^2c, b^2c^2, abcd, bc^2d, c^2d^2, abce, a^2be, abde, bcde, cd^2e, acde, ade^2, a^2e^2, d^2e^2.
\]
By using $e_1, e_2, e_3, e_4, e_5$, the above ordering is 
\begin{eqnarray}
\label{Istanbul}
e_1^2, e_1e_2, e_2^2, e_1e_3, e_2e_3, e_3^2, e_2e_5, e_1e_5, e_1e_4, e_2e_4, e_3e_4, e_3e_5, e_4e_5, e_5^2, e_4^2.
\end{eqnarray}
The ordering of pure powers is
\[
e_1^2 > e_2^2 > e_3^2> e_5^2 > e_4^2.
\]

We show that Conjecture \ref{NewOrleans} is true for the pentagon with $q_0 = 2$.

\begin{Proposition}
\label{pentagon}
All powers $I(C_5)^q$ with $q \geq 2$ have linear quotients.
\end{Proposition}

\begin{proof}
Starting from the linear quotient ordering (\ref{Istanbul}) of $I(C_5)^2$, we prove that Conjecture \ref{NewOrleans} is true for $q \geq 2$ by using induction on $q$.  Note that $e_1=ab,e_2=bc,e_3=cd,e_4=de$ and $e_5=ea$ are algebraically independent in $S=K[a,b,c,d,e]$.  

Let $q > 3$ and suppose that $I(C_5)^{q-1}$ has linear quotients with the ordering $u_1>\cdots >u_{\binom{5+(q-1)-1}{q-1}}$ coming from Conjecture \ref{NewOrleans}.  In particular $u_1 = e_1^{q-1}$ and $u_{\binom{5+(q-1)-1}{q-1}} = e_4^{q-1}$.  Let $w_1 > \cdots > w_{\binom{5+q-1}{q}}$ denote the ordering of $I(C_5)^{q}$ coming from Conjecture \ref{NewOrleans}. 
   
Let $\Ac(e_1)$ denote the set of $w_i$ which  can be written as a product of $q$ edges, one of which is $e_1$. In this case, we say $w_i$ has a presentation using $e_1$.  Let
$\Ac(e_2)$ denote the set of $w_i \not\in \Ac(e_1)$ which have a presentation using $e_2$, $\Ac(e_3)$ the set of $w_i \not\in \Ac(e_1) \cup \Ac(e_2)$ which have a presentation using $e_3$ and $\Ac(e_5)$ the set of $w_i \not\in \Ac(e_1) \cup \Ac(e_2) \cup \Ac(e_3)$ which have a presentation using $e_5$.  Finally, $\Ac(e_4) = \{e_4^q\}$.

We claim that for any $2 \leq p \leq \binom{5+q-1}{q}$, the colon ideal $(w_1, \ldots, w_{p-1}):w_p$ is generated by a subset of $\{a,b,c,d,e\}$.  
If $w_p \in \Ac(e_1)$, then the claim follows from the induction hypothesis.  

\noindent\textbf{CASE 1:} Let $w_p \in \Ac(e_2)$ and $w_i \in \Ac(e_2)$ with $i < p$.  Let $w_i = u_{\ell}e_2$ and $w_p = u_te_2$ with $s < t$.  Then there is $u_\ell$ with $\ell < t$ for which $u_\ell : u_t$ is a variable and $(u_\ell : u_t)|(u_{\ell} : u_t)$. Set $w_j = u_\ell e_2$ and notice that $j < p$.  Note that $w_j$ may belong to $\Ac(e_1)$.  Then $w_j : w_p$ is a variable and $(w_j : w_p)|(w_i : w_p)$, as desired.  

\noindent\textbf{CASE 2:} Let $w_p \in \Ac(e_2)$ and $w_i \in \Ac(e_1)$.  If $a$ does not divide $w_p$, then $w_i:w_p$ is divisible by $a = (e_1w_p/e_2):w_p$.  If $e_5 = ae$ appears in the presentation of $w_p$, then both $(e_1 w_p / e_5):w_p = b$ and $(e_1e_3 w_p / e_2 e_5):w_p = d$  
belong to $(w_1, \ldots, w_{p-1}):w_p$.  Let $w_i = e_1^{f_1}e_2^{f_2}e_3^{f_3}e_4^{f_4}e_5^{f_5}$ and $w_p = e_2^{f'_2}e_3^{f'_3}e_4^{f'_4}e_5^{f'_5}$.  If one of the inequalities
\[
f_1 + f_5 > f'_5, \, \, \, f_1 + f_2 > f_2', \, \, \, f_3 + f_4 > f'_3 + f'_4 
\]
 holds, then $w_i : w_p \in (a,b,d)$.  If none of the above inequalities is satisfied, then
\[
2f_1 + f_2 + f_3 + f_4 + f_5 \leq f_2' + f_3' + f'_4 +f'_5,
\]
which is impossible, since $\sum_{i=1}^5 f_i = \sum_{i=2}^5f'$ and $f_1 > 0$.

\noindent\textbf{CASE 3:} Let $w_p \in \Ac(e_3)$.  Since $e_2(w_p/e_3) \in \Ac(e_2)$, it follows that $(e_2w_p/e_3):w_p = b $ belongs to $(w_1, \ldots, w_{p-1}):w_p$.  Since $b$ does not divide $e_p$, if $w_i \in \Ac(e_1) \cup \Ac(e_2)$, then $b$ divides $w_i:w_p$.  If $w_i \in \Ac(e_3)$, then a similar argument as in Case 1
is valid without modification. 

\noindent\textbf{CASE 4:} Let $w_p \in \Ac(e_5)$.  Since $e_1(w_p/e_5) \in \Ac(e_1)$, it follows that $(e_1w_p/e_5):w_p = b $ belongs to $(w_1, \ldots, w_{p-1}):w_p$.  Since $b$ does not divide $e_p$, if $w_i \in \Ac(e_1) \cup \Ac(e_2)$, then $b$ divides $w_i:w_p$.  If $e_4$ appears in the presentation of $w_p$, then $e_3(w_p/e_4)\in \Ac(e_3)$ and $(e_3w_p/e_4):w_p = c $ belongs to $(w_1, \ldots, w_{p-1}):w_p$.  Since $c$ does not divide $e_p$, if $e_i \in \Ac(e_3)$, then $c$ divides $w_i:w_p$.  
Let $w_p = e_5^q$.  Since $e_4e_5^{q-1} > e_5^q$ and $e_4e_5^{q-1} :e_5^q = d$, then $d$ belongs $(w_1, \ldots, w_{p-1}):w_p$.  Since $d$ does not divide $w_p = e_5^q$, if $e_i \in \Ac(e_3)$, then $d$ divides $w_i:w_p$.  If $e_i \in \Ac(e_5)$, then a similar argument as in Case 1
is valid without modification. 

\noindent\textbf{CASE 5:} Finally, if $w_p \in \Ac(e_4)$, in other words, $e_p = e_4^q$, then a routine computation done as above yields $(w_1, \ldots, w_{p-1}):w_p = (a,c)$.  This completes the proof. 
\end{proof}

\begin{Remark}
The following ordering
\[
a^2b^2,a^2be,ab^2c,abce,a^2e^2,b^2c^2,abcd,abde,bc^2d,ade^2,bcde,acde,cd^2e,c^2d^2,d^2e^2.
\]
can also be an ordering for $I(C_5)^2$ to have linear quotients, where the ordering of pure powers is
$e_1^2 > e_5^2 > e_2^2> e_3^2 > e_4^2$.
\end{Remark}


\section{Pentagon together with one vertex} \label{sec.Pentagon}

Developing the argument from the proof of Proposition \ref{pentagon}, we now study finite graphs formed by carefully adding vertices to the pentagon. We first study the case where a single vertex is added with specific connecting edges.

Let $G$ be
the finite graph with the vertices $\{a,b,p,q,x,z\}$ and with edges 
\begin{eqnarray*}
& e_1=\{a,b\},e_2=\{a,x\},e_3=\{b,x\},e_4=\{a,p\}, e_5=\{b,q\}, &
\\
& e_6 = \{p,z\}, e_7 = \{x,z\}, e_8 = \{q,z\} &
\end{eqnarray*}
depicted below (Figure $2$). 
\begin{figure}[h]
\centering
\begin{tikzpicture}[scale=1.5]
\coordinate (z) at (0,1){};
\coordinate (p) at (-1,0.3){};
\coordinate (q) at (1,0.3){};
\coordinate (x) at (0,0){};
\coordinate (a) at (-0.7,-0.7){};
\coordinate (b) at (0.7,-0.7){};
\fill(z)circle(0.7mm);
\fill(p)circle(0.7mm);
\fill(q)circle(0.7mm);
\fill(x)circle(0.7mm);
\fill(a)circle(0.7mm);
\fill(b)circle(0.7mm);
\draw(z)--(p)--(a)--(b)--(q)--cycle;
\draw(z)--(x)--(a);
\draw(x)--(b);
\draw(0,1.2)node{{\Large $z$}};
\draw(-1.2,0.3)node{{\Large $p$}};
\draw(1.2,0.3)node{{\Large $q$}};
\draw(0.2,0.1)node{{\Large $x$}};
\draw(-0.7,-0.9)node{{\Large $a$}};
\draw(0.6,-0.9)node{{\Large $b$}};
\end{tikzpicture}
\caption{Pentagon together with one vertex}
\label{figure:susan_selvi}
\end{figure}
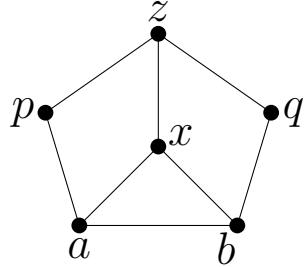
The edge ideal $I(G)$ has $8$ generators and, since $e_2e_6=e_4e_7$ and $e_3e_8 = e_5e_7$, it follows that $I(G)^2$ has $\binom{8+1}{2} -2 = 36-2=34$ generators.  
Let $\Nc^{(2)}$ denote the following order on these generators:
\[
\begin{array}{cccc}
e_1^2=(ab)^2, & e_1e_2 = (ab)(ax), & e_1e_3=(ab)(bx), & e_1e_4=(ab)(ap), \\
e_1e_5 = (ab)(bq), & e_2e_4=(ax)(ap), & e_2e_3=(ax)(bx), & e_3e_5=(bx)(bq), \\ 
e_3e_4=(bx)(ap) & e_2e_5 = (ax)(bq), & e_4e_5=(ap)(bq), & e_1e_6=(ab)(pz), \\ 
e_1e_8=(ab)(qz), & e_1e_7=(ab)(xz), & e_4e_7=(ap)(xz), & e_5e_7 = (bq)(xz),\\ 
e_3e_6=(bx)(pz), & e_3e_8=(ax)(qz), & e_2e_7=(ax)(xz) & e_3e_7 = (bx)(xz), \\
e_4e_6 = (ap)(pz), & e_5e_8=(bq)(qz), & e_4e_8=(ap)(qz), & e_5e_6=(bq)(pz), \\ 
e_6e_7=(pz)(xz), & e_6e_8=(pz)(qz), & e_7e_8=(xz)(qz), & e_4^2=(ap)^2, \\ 
e_2^2 = (ax)^2, & e_3^2=(bx)^2, & e_5^2 = (bq)^2, & e_6^2 = (pz)^2, \\
e_7^2 = (xz)^2, & e_8^2 = (qz)^2. & & \\
\end{array}
\]
It is now a direct computation to verify that with the order $\Nc^{(2)}$, $I(G)^2$ has linear quotients.

Using the ordering $\Nc^{(2)}$ of the generators of $I(G)^2$ as above,
one can inductively show that all powers of $I(G)$ have linear quotients. The following  notation will be useful in the proof of Proposition \ref{p:SusanSelvi}.

\begin{Notation}
Let $G$ be a finite graph and $w$ be one of the generators of $I(G)^s$ for $s\geq 1$.  The notation $e_j \shortmid w$ means that there exists a factorization $w=e_{i_1}\cdots e_{i_N}$ with $e_j = e_{i_k}$ for some $1 \leq k \leq s$.
\end{Notation}

Note that this notation is intended to convey a notion of division in terms of a factorization into a product of generators, rather than literal division. For instance, if $G$ is as in Figure \ref{figure:susan_selvi}, then $e_1 = ab$ divides $w=e_4e_5=abpq$, but $e_1 \nshortmid w$ since $pq$ is not a generator of $I(G)$.

\begin{Proposition}
\label{p:SusanSelvi}
Let $G$ be the finite graph in Figure $2$.  Then all powers $I(G)^s$ with $s \geq 2$ have linear quotients.
\end{Proposition}

\begin{proof}
We prove that Conjecture \ref{NewOrleans} is true for $G$ with $q_0 = 2$.
Assume $s \geq 3$ and suppose that $I(G)^{s-1}$ has linear quotients with the ordering $\Nc^{(s-1)}$: $u_1>\cdots >u_{r}$  where  $u_1 = e_1^{s-1}$ and $ u_{r} = e_8^{s-1}$.  Let $\Nc^{(s)}: ~~w_1 > \cdots > w_{t}$ denote the efficient ordering of $I(G)^{s}$ constructed from $\Nc^{(s-1)}$ as in Conjecture \ref{NewOrleans}. For ease of notation,  use the following notation to refer to elements in $\Nc^{(s)}$ 
    $$\Ac(e_1) , \Ac(e_4), \Ac(e_2), \Ac(e_3), \Ac(e_5), \Ac(e_6), \Ac(e_7), \Ac(e_8)$$
    where $w_k \in \Ac(e_i)$ means $e_i \shortmid w_k$ and $e_j \nshortmid w_k$ for for $j$ such that $e_j^{s-1} > e_i^{s-1}$ in $\Nc$.

    We claim that the colon ideal $(w_1,\ldots, w_{n-1}):w_n$ is linear for each $w_n$. If $w_n\in \Ac(e_1)$, then the claim follows from the induction hypothesis. 
    
    \noindent\textbf{CASE 1:} Let  $w_n =u_{\ell}e_4\in \Ac(e_4)$  for $\ell \leq r$. First, observe that $(b) \subseteq (w_1,\ldots, w_{n-1}):w_n$ since $(u_{\ell} e_1: u_{\ell} e_4) = (b)$ for $u_{\ell} e_1\in \Ac(e_1)$.  Let $w_i= u_te_4 \in \Ac(e_4)$  for $t<\ell$. By the induction hypothesis,  there exists $u_l$ with $l<\ell$ such that $u_l:u_{\ell}$ is a variable dividing $u_t:u_{\ell}$. Let $w_j=u_le_4$ and notice that $j<n$. Then, $w_j:w_n$ is a variable that divides $w_i:w_n$, as desired. Note that $w_j \in  \Ac(e_1)$ is possible.

    Next, let  $w_i \in \Ac(e_1)$.
    Recall that $b \nmid e_4$.  If  $b \nmid u_{\ell}$,  then   $w_i:w_n\in (b)$ since every element of $\Ac(e_1)$ is divisible by $b$. Then,  $(u_1e_1,\ldots,u_re_1):w_n = (b)$. Hence,  $(w_1,\ldots, w_{n-1}):w_n$ is linear in this case. Suppose $b \mid u_{\ell}$. Then either $e_3 \shortmid u_{\ell}$ or $e_5 \shortmid u_{\ell}$. If $e_3 \shortmid u_{\ell}$, then $u_{\ell}= w' e_3$ for a generator $w' \in I(G)^{s-2}$. Consider the elements $w_j= w' e_4 e_1$ and $w_k= w' e_7e_1$ in $\Ac(e_1)$. Note that $(w_j: w_n)= (a) $ and $  (w_k:w_n)= (z)$. 
    If $e_5 \shortmid u_{\ell}$, then $u_{\ell}= w'' e_5$ for a generators $w'' \in I(G)^{s-2}$. Similarly, consider $w_j= w'' e_4 e_1$ and $w_k= w'' e_8e_1$ in $\Ac(e_1)$ where $(w_j: w_n)= (a) $ and $  (w_k:w_n)= (z)$.  So, in either of these scenarios, we have
     $(a,b, z) \subseteq (w_1,\ldots, w_{n-1}):w_n.$
    Let $w_i = e_1^{t_1}\cdots e_8^{t_8}$ where $t_1>0$ and $w_n = e_2^{t_2'}\cdots e_8^{t_8'}$ since $e_1 \nshortmid w_n$. If $w_i:w_n \notin (a,b,z)$, then examining the degrees of $a,b,z$ in $w_i$ and $w_n$, respectively, results in 
    \begin{eqnarray*}
    t_1 + t_2 + t_4 &\leq& t_2' + t_4'\\
    t_1+t_3+t_5 &\leq& t_3'+t_5'\\
    t_6+t_7+t_8 &\leq& t_6'+t_7' +t_8'
\end{eqnarray*}
which implies that $t_1\leq 0$, a contradiction. Hence, $(w_1,\ldots, w_{n-1}):w_n$ is linear for each $w_n \in \Ac(e_1) \cup \Ac(e_4)$.

\noindent\textbf{CASE 3:} Let $w_n=u_{\ell}e_2 \in \Ac(e_2)$ for some $\ell \leq r$.  Note that $e_6 \nshortmid w_n$ because $e_6e_2 =e_7e_4$ and $e_4 \nshortmid e_n$. Next,  observe that  $(b,p) \subseteq (w_1,\ldots, w_{n-1}):w_n$ since $(u_{\ell} e_1 : u_{\ell} e_2) =(b)$ and $(u_{\ell} e_4: u_{\ell} e_2) = (p)$. Let $w_i \in \Ac(e_2)$. As in the previous case, by the induction hypothesis,  there exists $w_j$ for $j<n$  such that $w_j:w_n$ is a variable that divides $w_i:w_n$, as desired.  Let $w_i \in \Ac(e_1) \cup \Ac(e_4)$.  Observe that $p \nmid u_{\ell}$; otherwise,   $e_4 \shortmid u_{\ell}$ or $e_6 \shortmid u_{\ell}$ but neither of these cases are possible since $w_n \in \Ac(e_2)$. If   $b \nmid u_{\ell}$, then $(w_i:w_n) \subseteq (b,p)$ which implies that $(w_1,\ldots, w_{n-1}):w_n$ is generated by variables. So, claim holds for this case. 
Suppose $b \mid u_{\ell}.$ Then either $e_3\shortmid u_{\ell}$ or $e_5\shortmid u_{\ell}$ since $e_1 \nshortmid u_{\ell}.$ If $e_3 \shortmid u_{\ell}$, then $u_{\ell}=w'e_3=w'bx$ for some $w' \in I(G)^{s-2}$.  Consider the elements $w_j= w'e_2e_1 $ and $w_k= w'e_7e_1$ in $ \Ac(e_1)$ where $(w_j:w_n)=(a)$ and $(w_k:w_n)=(z)$. So, 
$(a,b,p, z) \subseteq (w_1,\ldots, w_{n-1}): w_n.$

Assuming $e_5 \shortmid u_{\ell}$ also yields to the above containment. Let $w_i = e_1^{t_1}\cdots e_8^{t_8}$  where $t_1>0$ or $t_4 >0$ and  $w_n = e_2^{t_2'}e_3^{t_3'}e_5^{t_5'}e_t^{t_7'}e_8^{t_8'}$ since $e_1, e_4, e_6$ do not divide $w_n=u_{\ell}e_2.$  If $w_i : w_n \notin  (a,b,p,z)$, then  examining degrees of $a,b,p,z$ in $w_i$ and $w_n$ results with
 $t_1+t_4+t_6\leq 0$, a contradiction.  So, $(w_1,\ldots, w_{n-1}):w_n$ is linear in this case.

\noindent\textbf{CASE 4:} Now, let $w_n=u_{\ell}e_3 \in \Ac(e_3)$ for some $u_{\ell}\in I(G)^{s-1}$.  First observe that  $(a) \subseteq (w_1,\ldots, w_{n-1}):w_n$ since $(u_{\ell} e_1 : u_{\ell} e_3) =(a)$.  The case $w_i \in \Ac(e_3)$ follows from the induction hypothesis as in the previous cases.    Let $w_i \in \Ac(e_1) \cup \Ac(e_4) \cup \Ac(e_2)$. Note that this collection contains all generators of $I(G)^s$ that are divisible by $a$ since $e_1, e_2$ and $e_4$ are the only edges incident to $a$ and $a \nmid w_n$.  This means  $w_i:w_n \in (a)$. Hence, $(w_1,\ldots, w_{n-1}):w_n$ is linear in this case.

 Next, let $w_n=u_{\ell}e_5 \in \Ac(e_5)$ for some $u_{\ell}\in I(G)^{s-1}$.  First observe that  $(a,x) \subseteq (w_1,\ldots, w_{n-1}):w_n$ since $(u_{\ell} e_1 : u_{\ell} e_5) =(a)$ and $(u_{\ell} e_3 : u_{\ell} e_5) =(x)$.  The case $w_i \in \Ac(e_5)$ follows from the induction hypothesis as in the previous cases.  Let $w_i \in \Ac(e_1) \cup \Ac(e_4) \cup \Ac(e_2)\cup \Ac(e_3)$.  Note that $w_i$ is divisible by $a$ or $x$ and $w_n$ is not divisible by $a$. Furthermore, $w_n$ is not divisible by $x$. Otherwise, if $x \mid u_{\ell}$, then $u_{\ell}=w' e_7$ for a generator $w'$ of $I(G)^{s-2}$ and $w_n=w' e_7 e_5 = w' e_6 e_3$ which implies that $e_3 \shortmid w_n$, a contradiction. So, $w_i:w_n \in(a,x)$.  Hence, $(w_1,\ldots, w_{n-1}):w_n$ is linear in this case.

\noindent\textbf{CASE 5:} Let $w_n=u_{\ell}e_6 \in \Ac(e_6)$ for some $u_{\ell}\in I(G)^{s-1}$.  Observe that  $(a) \subseteq (w_1,\ldots, w_{n-1}):w_n$ since $(u_{\ell} e_4 : u_{\ell} e_6) =(a)$.  The case $w_i \in \Ac(e_6)$ follows from the induction hypothesis as in the previous cases.  Let $w_i \in \Ac(e_1) \cup \Ac(e_4) \cup \Ac(e_2)\cup \Ac(e_3) \cup \Ac(e_5)$.  Note that $w_i$ is divisible by $a$ or $b$ while $a\nshortmid  w_n$ and  $b \nshortmid  w_n$. This implies that $w_i:w_n \in (a,b)$.

Observe that $u_{\ell}= e_6^{t_6}e_7^{t_7}e_8^{t_8}$ since $ u_{\ell}$ must be divisible by $z$. If $t_8 \geq 1$, then $u_{\ell} = w' e_8$ for a generator $w' \in I(G)^{s-2}$ and $(w' e_6 e_5 : w' e_6 e_8)= (b)$. Similarly, if $t_7\geq 1$, then $u_{\ell}= w' e_7$ for a generator $w' \in I(G)^{s-2}$ and $(w' e_6 e_3 : w' e_6 e_7)= (b)$. So, $(w_1,\ldots, w_{n-1}):w_n$ is linear in these cases. For the remaining case, we have $w_n= e_6^s$. Notice that, for each $w_i$, we have  $w_i:w_n \in (a,x,q)$ since $\{a,x,q\}$ is a  vertex cover of $G\setminus \{e_6\}$. Furthermore, note that the generators $w_j=e_6^{s-1}e_7 $ and $w_k= e_6^{s-1}e_8$ come before $w_n$ in $\Ac(e_6)$  from  the construction of $\Nc^{(s)}$. In addition, $(w_j:w_n) = (x)$ and  $(w_k:w_n)= (q)$.   Thus,  the claim holds.

\noindent\textbf{CASE 6:} For the final two cases, if $w_n=u_{\ell}e_7 \in \Ac(e_7)$ then none of $a,b,p$ divide $u_{\ell},$ but every element of $\Ac(e_1) \cup \Ac(e_4) \cup \Ac(e_2) \cup \Ac(e_3) \cup \Ac(e_5)\cup  \Ac(e_6)$ is divisible by one of these three variables. Since $(u_{\ell} e_2:u_{\ell} e_7)=(a)$ and $(u_{\ell} e_3:u_{\ell} e_7)=(b)$ and $(u_{\ell} e_6:u_{\ell} e_7)=(p)$, we are done. For the last case, we have $w_n=e_8^s$. Note that $(e_8^{s-1}e_5:e_8^s)=(b), ~~(e_8^{s-1}e_6:e_8^s)=(p)$ and $(e_8^{s-1} e_7:e_8^s)=(x)$. Since $\{b,p,x\}$ is a vertex cover of $G \setminus \{e_8\}$, we have $w_i : w_n \in (b,p,x)$ for each  $w_i \neq w_n$. Therefore, the claim holds. 
\end{proof}

As an example of how \Cref{t:complete} can be applied, consider the class of graphs formed by replacing the central vertex $x$ of~\Cref{figure:susan_selvi} by a complete graph $K_n$, with each vertex of $K_n$ connected to $a,b,z$. 
We denote such a graph by $(C_5,K_n)$ and  label the vertices of $K_n$ by $x_1, \ldots,x_n$. The case $n=4$ is shown below.
\begin{figure}[h]
\centering
\begin{tikzpicture}[scale=1.5]
\coordinate (z) at (0,1){};
\coordinate (p) at (-1,0.3){};
\coordinate (q) at (1,0.3){};
\coordinate (x1) at (-.18,-.26){};
\coordinate (x2) at (-.26,.2){};
\coordinate (x4) at (.26,-.2){};
\coordinate (x3) at (.22,.24){};
\coordinate (a) at (-0.7,-0.7){};
\coordinate (b) at (0.7,-0.7){};
\fill(z)circle(0.7mm);
\fill(p)circle(0.7mm);
\fill(q)circle(0.7mm);
\fill(x1)circle(0.7mm);
\fill(x2)circle(0.7mm);
\fill(x3)circle(0.7mm);
\fill(x4)circle(0.7mm);
\fill(a)circle(0.7mm);
\fill(b)circle(0.7mm);
\draw(z)--(p)--(a)--(b)--(q)--cycle;
\draw(z)--(x1)--(a);
\draw(z)--(x2)--(a);
\draw(z)--(x3)--(a);
\draw(z)--(x4)--(a);
\draw(x1)--(b);
\draw(x2)--(b);
\draw(x3)--(b);
\draw(x4)--(b);
\draw(x1)--(x2)--(x3)--(x4)--cycle;
\draw(x1)--(x3);
\draw(x2)--(x4);
\draw(0,1.2)node{{\Large $z$}};
\draw(-1.2,0.3)node{{\Large $p$}};
\draw(1.2,0.3)node{{\Large $q$}};
\draw(-.2,-.5)node{$x_1$};
\draw(-.5,.2)node{$x_2$};
\draw(.5,.2)node{$x_3$};
\draw(.5,-.2)node{$x_4$};
\draw(-0.7,-0.9)node{{\Large $a$}};
\draw(0.7,-0.9)node{{\Large $b$}};
\end{tikzpicture}
\caption{$(C_5,K_4)$}
\label{figure:pentagon-complete}
\end{figure}

\begin{Example}
Let $I$ be the edge ideal of the graph $(C_5,K_n)$ for some $n\geq 1$. Then $I^s$ has linear quotients for all $s \ge 2$. The case $n=1$ holds by \Cref{p:SusanSelvi}. For $n>1$ the result holds via induction using Proposition \ref{p:duplication}.
\end{Example}

Notice that the above argument cannot extend to replace any vertex of the exterior pentagon with an edge. In particular, this would result in a finite graph that is not gapfree. However, the technique can be applied to the central vertex in \Cref{figure:selvi}, for instance, which is the base case for a second class of gapfree graphs built from adding a vertex to a pentagon.

\begin{figure}[h]
\centering
\begin{tikzpicture}[scale=1.5]
\coordinate (z) at (0,1){};
\coordinate (p) at (-1,0.3){};
\coordinate (q) at (1,0.3){};
\coordinate (x) at (0,0.-0.05){};
\coordinate (a) at (-0.7,-0.7){};
\coordinate (b) at (0.7,-0.7){};
\fill(z)circle(0.7mm);
\fill(p)circle(0.7mm);
\fill(q)circle(0.7mm);
\fill(x)circle(0.7mm);
\fill(a)circle(0.7mm);
\fill(b)circle(0.7mm);
\draw(z)--(p)--(a)--(b)--(q)--cycle;
\draw(p)--(x)--(q);
\draw(x)--(a);
\draw(x)--(b);
\draw(0,1.2)node{{\Large $z$}};
\draw(-1.2,0.3)node{{\Large $p$}};
\draw(1.2,0.3)node{{\Large $q$}};
\draw(0.02,0.15)node{{\Large $x$}};
\draw(-0.7,-0.9)node{{\Large $a$}};
\draw(0.6,-0.9)node{{\Large $b$}};
\end{tikzpicture}
\caption{Pentagon together with one vertex and multiple inner edges}
\label{figure:selvi}
\end{figure}
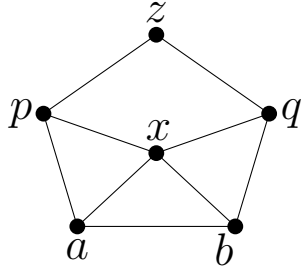

\begin{Proposition}\label{l:pentagon_4inner_edges}
Let $G$ be
the finite gapfree graph given in Figure \ref{figure:selvi}. Then $I(G)^s$ has linear quotients for all $s\geq 2$.
\end{Proposition}

\begin{proof}
Label the edges of $G$ as follows:
\begin{eqnarray*}
& e_1=\{a,b\}, ~~e_2=\{a,p\}, ~~e_3=\{a,x\}, ~~e_4=\{b,x\}, ~~e_5=\{b,q\}, &
\\
& e_6 = \{x,p\}, ~~e_7 = \{x,q\}, ~~e_8 = \{p,z\}, ~~e_9=\{q,z\} &
\end{eqnarray*} 
One can verify that  $I(G)^2$ has linear quotients with respect to the following order:
    \[
\begin{array}{cccc}
e_1^2=(ab)^2, & e_1e_3 = (ab)(ax), & e_1e_4=(ab)(bx), & e_1e_2=(ab)(ap), \\
e_1e_6=(ab)(xp), &  e_1e_5=(ab)(bq), & e_1e_7=(ab)(xq), &  e_1e_8=(ab)(pz),\\
e_1e_9=(ab)(qz), & e_2e_8=(ap)(pz), & e_2e_3=(ap)(ax), & e_2e_7=(ap)(xq),\\
e_2e_9=(ap)(qz), & e_2e_5=(ap)(bq), &e_2e_6=(ap)(xp), & e_3e_8=(ax)(pz),\\
e_3e_4=(ax)(bx), & e_3e_7=(ax)(xq), & e_3e_9=(ax)qz), & e_3e_6=(ax)(xp),\\
e_4e_8=(bx)(pz), & e_4e_7=(bx)(xq), & e_4e_9=(bx)(qz), & e_4e_5=(bx)(bq),\\
e_4e_6=(bx)(xp), & e_5e_9=(bq)(qz), & e_5e_6=(bq)(xp), & e_5e_8=(bq)(pz),\\
e_5e_7=(bq)(xq), & e_6e_9=(xp)(qz), & e_6e_8=(xp)(pz), & e_6e_7=(xp)(xq),\\
e_7e_9=(xq)(qz), & e_8e_9=(pz)(qz), & e_2^2= (ap)^2, & e_3^2=(ax)^2,\\
e_4^2=(bx)^2, & e_5^2=(bq)^2, & e_6^2= (xp)^2, & e_7^2=(xq)^2,\\
e_8^2=(pz)^2, & e_9^2=(qz)^2. & &
\end{array}
\]
Assume $s \geq 3$. Suppose that $I(G)^{s-1}$ has linear quotients with the ordering $\Nc^{(s-1)}$: $u_1>\cdots >u_{r}$  where  $u_1 = e_1^{s-1}$ and $ u_{r} = e_9^{s-1}$.  Let $\Nc^{(s)}: ~~w_1 > \cdots > w_{t}$ denote the efficient ordering of $I(G)^{s}$ constructed from $\Nc^{(s-1)}$ as in  Conjecture \ref{NewOrleans}. As before,  use the following notation to refer to elements in $\Nc^{(s)}$ 
    $$\Ac(e_1) , \Ac(e_2), \Ac(e_3), \Ac(e_4), \Ac(e_5), \Ac(e_6), \Ac(e_7), \Ac(e_8), \Ac(e_9),$$
    where $w_k \in \Ac(e_i)$ means $e_i \shortmid w_k$ and $e_j \nshortmid w_k$ for each $j<i$. We claim that the colon ideal $(w_1,\ldots, w_{n-1}):w_n$ is linear for each $w_n$. If $w_n\in \Ac(e_1)$,  the claim follows from the induction hypothesis.

   \noindent\textbf{CASE 1:} Let $w_n= u_{\ell}e_2 \in \Ac(e_2)$ for some $\ell \leq r$. Note that $(b) \subseteq (w_1,\ldots, w_{n-1}):w_n$ since  $(u_{\ell}e_1:u_{\ell}e_2)=(b)$. If $w_i\in \Ac(e_2)$, it follows from the induction hypothesis that there exists $j<n$ such that $w_j:w_n$ is a variable dividing $w_i:w_n$, as desired. Let $w_i\in \Ac(e_1)$. If $ b\nmid u_{\ell}$, then $w_i:w_n \in (b)$   since each element in $\Ac(e_1)$ is divisible by $b$. Thus, $(w_1,\ldots, w_{n-1}):w_n$ is linear and we are done. Suppose $b \mid u_{\ell}$. Then either $e_4$ or $e_5$ divides $u_{\ell}$ since $e_1 \nshortmid u_{\ell}$. Since $e_2e_4=e_1e_6$, it is not possible to have  $e_4 \shortmid u_{\ell}$ . So, $e_5 \shortmid u_{\ell}$ where $u_{\ell}=w' e_5$ for a generator $w' \in I(G)^{s-2}$.  In this case, it is important to note that $e_3 \nshortmid u_{\ell}$ as $e_3e_5=e_1e_7$.  Consider the elements $w_j= w'e_2e_1, w_k= w' e_6e_1$ and $w_l= w'e_8e_1$ in $\Ac(e_1)$ where $(w_j:w_n)= (a), (w_k:w_n)= (x)$ and $(w_l:w_n)=(z)$. This means  $(a,b,x,z)\subseteq (w_1,\ldots, w_{n-1}):w_n.$

     Set $w_i = e_1^{t_1} \cdots e_9^{t_9}$ where $t_1>0$ and $w_n= e_2^{t'_2} e_5^{t'_5} e_6^{t'_6} e_7^{t'_7} e_8^{t'_8} e_9^{t'_9}$. If $w_i:w_n \notin (a,b,x,z)$,  we obtain the following  inequalities by examining the degrees of these four variables in $w_i$ and $w_n$
    \begin{align*}
        t_1+t_2+t_3 &\leq t'_2\\
        t_1+t_4+t_5 & \leq t'_5\\
        t_3+t_4+t_6+t_7 & \leq t'_6+t'_7\\
        t_8+t_9 &\leq t'_8+t'_9
    \end{align*}
    which implies that $t_1+t_3+t_4\leq 0$, a contradiction. Thus, $(w_1,\ldots, w_{n-1}):w_n$ is linear for each $w_n \in \Ac(e_1)\cup \Ac(e_2)$.

 \noindent\textbf{CASE 2:}   Let $w_n =u_{\ell} e_3 \in \Ac(e_3)$ for some $\ell \leq r$.  First recall that $e_1, e_2 \nshortmid u_{\ell}$ and also note that $e_5 \nshortmid u_{\ell}$ since $e_3e_5=e_1e_7$.  The case $w_i \in \Ac(e_3)$ follows from the induction hypothesis as in the previous case. Let $w_i\in \Ac(e_1)\cup \Ac(e_2)$. Then, $(b,p) \subseteq (w_1,\ldots, w_{n-1}):w_n$ since $(u_{\ell}e_1:w_n)=(b)$ and $(u_{\ell}e_2:w_n)=(p)$. If $b\nmid u_{\ell}$ and $p\nmid u_{\ell}$, then $(w_i:w_n) \subseteq (b,p)$ implies that  $(w_1,\ldots, w_{n-1}):w_n$ is linear. So, the claim holds for this case. If $b \mid u_{\ell}$, then  $e_4\shortmid u_{\ell}$. Let  $u_{\ell}=w' e_4$ for a generator $w' \in I(G)^{s-2}$. Consider $w_j= w'e_3e_1$ and $w_k= e_7e_1$ in $\Ac(e_1)$ such that $(w_j:w_n)=(a)$ and $(w_k:w_n)=(q)$. This means $(a,b,p,q) \subseteq (w_1,\ldots, w_{n-1}):w_n$. One can verify that the same  containment holds for the remaining case $p\mid u_{\ell}$ (which is split into two cases: $e_6 \shortmid u_{\ell}$ or $e_8\shortmid u_{\ell}$).  
    
    As in the previous case, let $w_i= e_1^{t_1} \cdots e_9^{t_9}$ where either $t_1>0$ or $t_2>0$ and $w_n= e_3^{t'_3}e_4^{t'_4}e_6^{t'_6} e_7^{t'_7} e_8^{t'_8} e_9^{t'_9}$. If $w_i:w_n \notin (a,b,p,q)$, as before, considering degrees of these variables results with  
$t_1+t_2+t_5 \leq 0$, a contradiction. Thus,  $(w_1,\ldots, w_{n-1}):w_n$ is linear in this case as well.

\noindent\textbf{CASE 3:} Let $w_n=u_{\ell}e_4  \in \Ac(e_4)$ for some $\ell \leq r$.  First observe that  $(a) \subseteq (w_1,\ldots, w_{n-1}):w_n$ since $(u_{\ell} e_1 : u_{\ell} e_4) =(a)$.  The case $w_i \in \Ac(e_4)$ follows from the induction hypothesis.    Let $w_i \in \Ac(e_1) \cup \Ac(e_2) \cup \Ac(e_3)$. Note that this collection contains all generators of $I(G)^q$ that are divisible by $a$ and also $a \nmid w_n$.  This means  $w_i:w_n \in (a)$. Hence, $(w_1,\ldots, w_{n-1}):w_n$ is linear in this case.

\noindent\textbf{CASE 4:} Let $w_n=u_{\ell}e_5  \in \Ac(e_5)$ for some $\ell \leq r$.  Note that  $(a,x) \subseteq (w_1,\ldots, w_{n-1}):w_n$ since $(u_{\ell} e_1 : u_{\ell} e_5) =(a)$ and $(u_{\ell} e_4 : u_{\ell} e_5) =(x)$.  The case $w_i \in \Ac(e_5)$ follows from the induction hypothesis.    If $w_i \in \Ac(e_1) \cup \Ac(e_2) \cup \Ac(e_3)$, then $w_i:w_n \in (a)$ as in the previous case since  each such $w_i$ is divisible by $a$ and also $a \nmid w_n$. Let $w_i \in  \Ac(e_4)$. If $x\nmid w_n$, then $w_i:w_n \in (x)$ and the claim holds. If $x \mid w_n$, then either $e_6 \shortmid u_{\ell}$ or $e_7 \shortmid u_{\ell}$. We consider the case $e_6 \shortmid u_{\ell}$ and note that the other case follows similarly.  Let $u_{\ell}=w' e_6$ for a generator $w'\in I(G)^{s-2}$.  Consider the elements $w_j= w'e_4e_5$ and $w_k= w'e_4e_9$. Note that $j,k <n$ and $(w_j:w_n)=(b)$ and $(w_k:w_n)=(z)$. This means $(a,b,x,z) \subseteq (w_1,\ldots, w_{n-1}):w_n$. As in the previous cases, let $w_i=e_4^{t_4}e_5^{t_5}\cdots e_9^{t_9}$ where $t_4>0$ and $w_n=e_5^{t'_5}\cdots e_9^{t'_9}$. If $w_i:w_n \notin (b,x,z)$, then considering degrees of these variables results with $t_4\leq 0$, a contradiction.  Thus,  $(w_1,\ldots, w_{n-1}):w_n$ is linear.

\noindent\textbf{CASE 5:} Let $w_n=u_{\ell}e_6  \in \Ac(e_6)$ for some $\ell \leq r$.  Observe that  $(a,b) \subseteq (w_1,\ldots, w_{n-1}):w_n$ since $(u_{\ell} e_2 : u_{\ell} e_6) =(a)$ and $(u_{\ell} e_4 : u_{\ell} e_6) =(b)$.  The case $w_i \in \Ac(e_6)$ follows from the induction hypothesis.    For the remaining case, note that each $w_i$  is divisible by $a$ or $b$ while $w_n$ is not divisible by neither $a$ nor $b$.  This means  $w_i:w_n \in (a,b)$. Thus, $(w_1,\ldots, w_{n-1}):w_n$ is linear in this case.

\noindent\textbf{CASE 6:} Let $w_n=u_{\ell}e_7  \in \Ac(e_7)$ for some $\ell \leq r$.  Notice that $(a,b,p) \subseteq (w_1,\ldots, w_{n-1}):w_n$ since $(u_{\ell} e_3 : u_{\ell} e_7) =(a), (u_{\ell} e_4 : u_{\ell} e_7) =(b)$ and $(u_{\ell} e_6 : u_{\ell} e_7) =(p)$.  The case $w_i \in \Ac(e_7)$ follows from the induction hypothesis.    For the remaining case, note that $w_i$ is divisible by either $a$ or $b$ or $p$ while $w_n$ is not divisible by neither $a$ or $b$. Furthermore, $w_n$ is not divisible by $p$ because if $p\mid w_n$, then $e_8 \shortmid u_{\ell}$ which implies that $e_7e_8=e_6e_9$ divides $w_n$, a contradiction. Thus,   $w_i:w_n \in (a,b,p)$ and   $(w_1,\ldots, w_{n-1}):w_n$ is linear. Hence, the claim holds for this case.

\noindent\textbf{CASE 7:} Let $w_n=u_{\ell}e_8  \in \Ac(e_8)$ for some $\ell \leq r$.  Notice that $(a,x) \subseteq (w_1,\ldots, w_{n-1}):w_n$ since $(u_{\ell} e_2 : u_{\ell} e_8) =(a)$ and $(u_{\ell} e_6 : u_{\ell} e_8) =(x)$. The case $w_i\in \Ac(e_8)$ follows from the induction hypothesis. For the remaining case, notice that each $w_i$ is divisible by either $a, b$ or $x$ while $w_n$ is not divisible by any of these three variables. Observe that $u_{\ell}=e_8^{t_8} e_9^{t_9}$. If $t_9\geq 1$, then $u_{\ell}=w' e_9$ for a generator $w' \in I(G)^{s-2}$ and $(w'e_8e_5:w'e_8e_9)=(b)$. If $t_9=0$, then $w_n=e_8^s$. Notice that, for each $w_i$, we have $w_i:w_n \in (a,x,q)$ since $\{a,x,q\}$ is a vertex cover of $G\setminus \{e_8\}$. Consider the monomial $w_j= e_8^{s-1}e_9$ where $j<n$ by  the definition of the order $\Nc^{(s)}$. Since $(w_j:w_n)=(q)$, the claim holds.

\noindent\textbf{CASE 8:} For the final case, we have $w_n=e_9^s$. Observe that $(e_9^{s-1}e_8:e_9^s)=(p)$ and $(e_9^{s-1}e_5:e_9^s)=(b)$ and $(e_9^{s-1} e_7:e_9^s)=(x)$. Since $\{b,p,x\}$ is a vertex cover of $G \setminus \{e_9\}$, we have $w_i : w_n \in (b,p,x)$ for each  $w_i \neq w_n$.  
\end{proof}
Now, we introduce a particular new class of finite graphs.  

\begin{Definition}\label{def:nurselian graph}
A finite graph $\Gamma$ is called \emph{CDCC graph} if $\Gamma$ satisfies the following conditions:
\begin{itemize}
\item
$\Gamma$ is gapfree;
\item
$\Gamma$ contains (i) \textbf{c}ricket, (ii) \textbf{d}iamond, (iii) $\mathbf{C}_4$ and (iv) $\mathbf{C}_5$ as induced subgraphs.
\end{itemize}
\end{Definition}
Let $\Gamma$ be a CDCC graph. Since $\Gamma$ contains $C_5$, it follows that the edge ideal $I(\Gamma)$ does not have linear resolution.  Furthermore, since $\Gamma$ contains cricket, diamond and $C_4$, it is unclear if all powers $I(\Gamma)^t$ with $t \geq 2$ have linear resolution.

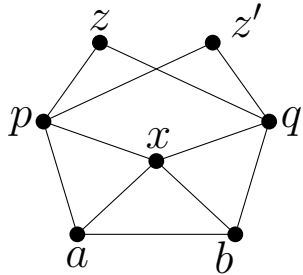
\begin{figure}[h]
\centering
\begin{tikzpicture}[scale=1.5]
\coordinate (z) at (-0.5,1){};
\coordinate (n) at (0.5,1){};
\coordinate (p) at (-1,0.3){};
\coordinate (q) at (1,0.3){};
\coordinate (x) at (0,0.-0.05){};
\coordinate (a) at (-0.7,-0.7){};
\coordinate (b) at (0.7,-0.7){};
\fill(z)circle(0.7mm);
\fill(n)circle(0.7mm);
\fill(p)circle(0.7mm);
\fill(q)circle(0.7mm);
\fill(x)circle(0.7mm);
\fill(a)circle(0.7mm);
\fill(b)circle(0.7mm);
\draw(z)--(p)--(a)--(b)--(q)--cycle;
\draw(p)--(x)--(q);
\draw(x)--(a);
\draw(x)--(b);
\draw(n)--(p);
\draw(n)--(q);
\draw(-0.5,1.2)node{{\Large $z$}};
\draw(0.8,1.2)node{{\Large $z'$}};
\draw(-1.2,0.3)node{{\Large $p$}};
\draw(1.2,0.3)node{{\Large $q$}};
\draw(0.02,0.15)node{{\Large $x$}};
\draw(-0.7,-0.9)node{{\Large $a$}};
\draw(0.6,-0.9)node{{\Large $b$}};
\end{tikzpicture}
\caption{A CDCC graph with minimum number of vertices}
\label{figure:oberwalfach}
\end{figure}

\begin{Example}
\label{Oberwolfach}
It is not difficult to see that no graph on $6$ vertices can satisfy all of the conditions in Definition~\ref{def:nurselian graph}. The graph $\Gamma$ displayed in Figure~\ref{figure:oberwalfach} is obtained from the graph in Figure~\ref{figure:selvi} by duplicating the vertex $z$. It is an example of a CDCC graph with smallest number of vertices. 
\end{Example}   

We now come to the next main theorem in the present paper.

\begin{Theorem}
\label{pasadenaAIM}
Given $n \geq 7$, there exists a CDCC graph $\Gamma_n$ with $n$ vertices for which all powers $I(\Gamma_n)^s$ with $s \geq 2$ have linear quotients.  
\end{Theorem}

\begin{proof}
Let $\Gamma_7$ denote the CDCC graph of Figure~\ref{figure:oberwalfach}. Since a finite graph obtained by duplicating a vertex from a CDCC graph is again CDCC graph, the desired result follows from Propositions \ref{p:duplication} and \ref{l:pentagon_4inner_edges}.
\, \, \, \, \, \, \, \, \, \, \, \, \, \, \, \, \, \, \, \, \, \, \, \, \, \, \, \, 
\end{proof}

While the matching number of the graph in Figure \ref{figure:oberwalfach} is $3$, in general, the matching number of a CDCC graph can be arbitrarily large. Indeed, duplicating both endpoints of an edge increases the matching number without changing the existing induced subgraphs and gapfreeness.


\section{A partial answer to Conjecture \ref{NewOrleans}} \label{sec.Partial}

In the present section, we consider a slightly modified version of Conjecture \ref{NewOrleans}. We show that under an additional assumption on the ordering, if $I^2, \, I^5, \, I^6$ and $I^7$ have linear quotients, then $I^q$ has linear quotients for all $q\geq 2$. Before stating the result, we introduce the concept of an admissible order. 

\begin{Definition} 
\label{def.goodOrder}
Let $G$ be a finite graph on the vertex set $V(G)$ and $E(G)$ the set of its edges.  An \emph{admissible} edge ordering of $G$ is a total order $\succ$ of the edges in $G$ with following property: for any $\{a, b\}, \{c, d\} \in E(G)$ such that $\{a, b\} \cap \{c, d\} = \emptyset$ and $\{a, b\} \succ \{c, d\}$, either all edges incident to $a$ are larger than $\{c, d\}$ or all edges incident to $b$ are larger than $\{c, d\}$.
\end{Definition}

Recall that, unless there is a misunderstanding, we use the notation $ab$ instead of $\{a, b\}$ for an edge of a finite graph. While not every order is admissible, such orders are ubiquitous in that every graph has one.

\begin{Lemma}
    \label{lem.GensOrder}
    Every finite graph possesses an admissible ordering of the edges.
\end{Lemma}

\begin{proof}
    We shall use induction on $n = |V(G)|$, the number of vertices in $G$.  The assertion is trivial if $n = 1$ or $n=2$. Suppose that $n \ge 3$ and the assertion is known to hold for any graph with fewer than $n$ vertices.  Let $x \in V(G)$ be any vertex in $G$, and set $G' = G \setminus \{x\}$. Suppose that the edges incident to $x$ in $G$ are $xx_1, \dots, xx_s$.

    By the induction hypothesis, there exists an admissible ordering of the edges in $G'$.  We shall order the edges in $G$ by $xx_1 \succ \dots \succ xx_s \succ \text{ edges in } G'$. We  show that this is an admissible ordering of the edges in $G$.  Consider any two edges $ab \succ cd$ with $ab \cap cd = \emptyset$ in this ordering. Clearly, $cd$ cannot be incident to $x$. 
    
    Consider the case that $ab$ is not incident to $x$. Then, both $ab$ and $cd$ are edges in $G'$. Thus, we can suppose that all edges incident to $a$ in $G'$ are larger than $cd$. If $a \not\in \{x_1, \dots, x_s\}$, then edges incident to $a$ in $G$ and $G'$ are the same, and we have the required property. If $a \in \{x_1, \dots, x_s\}$, then the only additional edge incident to $a$ in $G$ would be $xa$, which is larger than $cd$, by the edge ordering.

    Consider the remaining case that $ab$ is incident to $x$. Without loss of generality, we may assume that $a = x$. In this case, it is clear that all edges incident to $a$ are larger than $cd$.
\end{proof}

From now on, for a finite graph $G$, fix an admissible order of its edges:
    \begin{align}
    e_1 \succ \dots \succ e_s \label{eq.Compatible}\tag{$\ast$}
    \end{align}
The existence of this ordering is due to \Cref{lem.GensOrder}.

\begin{Definition}
    \label{def.NewOrder}
Let $G$ be a finite graph with edge ideal $I = I(G)$. Let 
    \begin{align}
        g_1 > \cdots > g_k
        \label{eq.I2}\tag{$\diamondsuit$}
    \end{align}
    be an order of the generators of $I^2$. We recursively construct the following orders for the generators of powers of $I$: if, for $q \ge 2$, 
    \begin{align*} 
    \Mc^{(q)}: u_1 > \dots > u_r
    \end{align*}
    denotes the constructed (or known) order of the generators of $I(G)^q$, then the order of the generators of $I(G)^{q+1}$ is given by
    \begin{align*} 
    \Mc^{(q+1)}: u_1e_1 > \dots > u_re_1 > u_1e_2 > \dots > u_re_2 > \dots > u_re_s, 
    \end{align*}
    where a generator $u_ie_j$ is omitted (\textit{redundant}) if it has already appeared before in the ordering. We call these orders $\Mc^{(q)}$ of the generators of $I^q$ for $q = 3,4,5,\ldots$ are the (\ref{eq.I2})--(\ref{eq.Compatible}) \textit{compatible} orders.
\end{Definition}

Note that the orders $\Mc^{(k)}$ given in Definitions \ref{def.NewOrder} look similar to the orders $\mathcal{N}^{(k)}$ described in Section \ref{sec.PureOrder}, but are different in nature. We now come to the last main result of the paper.

\begin{Theorem} \label{thm.LinQuotients}
    Let $G$ be a finite graph on the vertex set $V(G)$ and $E(G)$ the set of its edges.  Let $I = I(G)$ be the edge ideal of $G$.  
    Suppose that $I^2$ has linear quotients with an order of the generators written in the form of (\ref{eq.I2}). Suppose also that, for some integer $q \ge 7$, $I^2, \dots, I^q$ have linear quotients with the (\ref{eq.I2})--(\ref{eq.Compatible}) compatible orders of their generators. Then, $I^{q+1}$ also has linear quotients with the (\ref{eq.I2})--(\ref{eq.Compatible}) compatible order of its generators.
\end{Theorem}

\begin{proof}  Consider any two generators $A = u_ie_j > u_ke_\ell = B$ of $I^{q+1}$ in $\Mc^{(q+1)}$, and assume that they appear in the ordering $\Mc^{(q+1)}$ the first time in this representation; that is, $e_j$ and $e_\ell$ are the largest among all edges in the decomposition of $A$ and $B$, respectively. 
Set 
$$Q = A:B.$$
We need to show that there exists a generator $C$ of $I^{q+1}$ such that $C > B$ in $\Mc^{(q+1)}$, and $C:B$ is a variable that divides $Q$. 

Since $A > B$ in $\Mc^{(q+1)}$, we can assume that $e_j \succeq e_\ell$; that is, $j \le \ell$.
If $j = \ell$, then $u_i > u_k$ in $\Mc^{(q)}$, so there exists $u_t > u_k$ in $I^q$ such that $u_t : u_k$ is a variable that divides $Q = u_i: u_k$. Let $C = u_te_\ell$. Then, $C:u_ke_\ell = u_t:u_k$ is a variable that divides $Q$. Since $B$ appears in $\Mc^{(q+1)}$ first as $u_k e_\ell$ and $u_i > u_k$ in $\Mc^{(q)}$, it follows that $C > B$ in $\Mc^{(q+1)}$ and the assertion is proved in this case.

We can assume that $e_j \not= e_\ell$, i.e., $j \not= \ell$. Write $u_i = L_1 \cdots L_q$ and $u_k = M_1 \cdots M_q$, where $L_i$'s and $M_j$'s are edges in $G$, and set $e_j = ab, e_\ell = cd$. As observed before, $ab \succeq L_t$ and $cd \succeq M_t$ for all $t = 1, \dots, q$. Also, $ab \succ cd$. We shall analyze the following two possibilities: $ab \cap cd = \emptyset$ and $ab \cap cd \not= \emptyset$.
 \medskip
 
\noindent\textbf{CASE 1:} $ab \cap cd = \emptyset$. It can be seen that $ab \nmid u_k$. Indeed, if $ab \mid u_k$ then there must exist edges $ae = M_{t'}$ and $bf = M_{t''}$ in $u_k$, for some $t', t''$ (it is possible that $M_{t'} = M_{t''} = ab$). However, by the property of an admissible ordering, we must have either $ae \succ cd$ or $bf \succ cd$, which is a contradiction to the assumption that $M_t \preceq cd$ for all $t = 1, \dots, q$.

If $a \mid u_k$ and $b \nmid u_k$ (similarly for the case where $a \nmid u_k$ and $b \mid u_k$), then there exists an edge $ae$ in the presentation of $u_k$. Without loss of generality, we may assume that $ae = M_q$. Now, set 
$$C = (M_1 \cdots M_{q-1} \cdot cd) \cdot ab.$$
It can be seen that, since $ab \succ cd$, $C > B$ in $\Mc^{(q+1)}$. Moreover, $C:B = b$, which divides $Q$, as desired.

If $a \nmid u_k$ and $b \nmid u_k$ then set 
$$u_l = M_1 \cdots M_{q-1} \cdot ab.$$
Since $ab \succ cd \succeq M_q$, $u_l > u_k$ in $\Mc^{(q)}$. Therefore, by the linear quotients of $I^q$, there exists a generator $u_t > u_k$ of $I^q$ such that $u_t : u_k$ is a variable that divides $u_l : u_k$. Since $u_l : u_k = ab$, we have $u_t : u_k$ is either $a$ or $b$. Set 
$$C = u_t \cdot cd.$$
Since $B$ appears first in $\Mc^{(q+1)}$ in the form $u_k \cdot cd$, we have $C > B$. Moreover, $C:B$ is either $a$ or $b$, which divides $Q$ as desired.

\medskip

\noindent\textbf{CASE 2:} $ab \cap cd \not= \emptyset$. Since $ab \not= cd$, we may assume that $a = c$ and $b \not= d$. 

If $b \mid Q$, then set 
$$C = u_k \cdot ab.$$
Clearly, $C > B$ since $ab \succ ad$, and $C:B = b$ which divides $Q$.

If $a \mid Q$ and $b \nmid Q$, then there must exist an edge $be$ in $u_k$. Without loss of generality, assume that $be = M_q$ (particularly, $e \not= a$). Set
$$C = (M_1 \cdots M_{q-1}\cdot ad)\cdot ab.$$
Since $ab \succ ad \succeq be$, we have $C > B$. Also, $C:B = a$, which divides $Q$.

It remains to consider the case where $a \nmid Q$ and $b \nmid Q$. As before, since $b \nmid Q$, there must exists an edge $be$ in $u_k$, and we can assume that $be = M_q$.

Suppose that there exists an edge $uv \in E$ such that $uv \mid Q$. Without loss of generality, assume that $M_{q-1}$ is the largest among $M_1, \dots, M_{q-1}$. Set
$$u_g = (M_1 \cdots M_{q-2}) \cdot uv \cdot ab \, \, \, \text{ and } \, \, \, u_h = (M_1 \cdots M_{q-2}) \cdot be \cdot ad.$$
Again, since $ab \succ ad \succeq M_t$ for all $t$, we have 
$u_g > u_h \ge u_k$ in $\Mc^{(q)}$. By the linear quotients of $I^q$, there exists a generator $u_t > u_h$ of $I^q$ such that $u_t : u_h$ is a variable that divides $u_g : u_h = uv:de$ (which is necessarily not 1). Set
$$C = u_t \cdot M_{q-1}.$$
Then, $C:B = u_t:u_h$ is a variable which divides $uv \mid Q$. Hence, we are done in this case by the following Claim.

\medskip

\noindent\textsc{Claim 1.}
$C > u_h \cdot M_{q-1} = B$ in $\Mc^{(q+1)}$.

\noindent\textsc{Proof of Claim 1.} Since $u_t > u_h$, either $u_t$ contains an edge $f$ of $G$ for which $f \succ ad$ or $u_t = u_t' \cdot ad$, where $u_t'$ is a generator of $I^{q-1}$ and $u_t' > u_h/ad = M_1 \cdots M_{q-2} \cdot be.$ In the first case, due to the existence of $f$, we have $u_t \cdot M_{q-1} > u_h \cdot M_{q-1} = B$. Consider the case when $u_t = u_t' \cdot ad$, where $u_t' > (M_1 \cdots M_{q-2}) \cdot be = N$ in $\Mc^{(q-1)}$.

If $M_{q-1} > be$, then this implies that $u_t' \cdot M_{q-1} > N \cdot M_{q-1}$ in $\Mc^{(q)}$, and so $C = u_t' \cdot M_{q-1} \cdot ad > N \cdot M_{q-1} \cdot ad = u_h \cdot M_{q-1} = B$ in $\Mc^{(q+1)}$. If $M_{q-1} < be$, then either $u_t'$ contains an edge $f'$ of $G$ such that $f' \succ be \succ M_{q-1}$, or $u_t' = be \cdot u_t''$, where $u_t'' > M_1 \cdots M_{q-2}$ in $\Mc^{(q-2)}$. The existence of $f'$ gives $u_t' \cdot M_{q-1} > N \cdot M_{q-1}$ in $\Mc^{(q)}$ and, again, implies that $C = u_t' \cdot M_{q-1} \cdot ad > B$. On the other hand, if $u_t' = be \cdot u_t''$ as said, then since $M_{q-1} \succ M_i$ for all $i = 1, \dots, q-2$, we have $u_t'' \cdot M_{q-1} > (M_1 \cdots M_{q-2}) \cdot M_{q-1}$, which implies that $u_t' \cdot M_{q-1} = u_t'' \cdot M_{q-1} \cdot be > N \cdot M_{q-1}$. As before, it follows that $C > B$. The claim is proved.

We proceed by assuming now that an edge $uv$ such that $uv \mid Q$ does not exist. Particularly, it follows that $Q$ is the product of variables (with possibly repetitions) in an independent set of $G$.

Write $Q = x_1 \dots x_\gamma$, where the $x_l$'s (with possible repetition) form an independent set in $G$. By relabeling, if necessary, we may assume that $L_l = x_ly_l$, for $l = 1, \dots, \gamma$ (in particular, $\gamma \le q$). Since $y_l \nmid Q$, for $l = 1, \dots, \gamma$, we may also assume that $M_l = y_lz_l$, for $l = 1, \dots, \gamma$, or that some of the $y_l$'s are $d$ or $e$.

If $x_ld \in E$ (similarly, if $x_le \in E$) for some $l \in \{1, \dots, \gamma\}$, then set
$$C = (M_1 \cdots M_{q-1}) \cdot x_ld \cdot ab.$$
As observed before, since $ab \succ ad \succeq M_t$ for all $t$, $C > B$. Also, $C: B = x_l$, which divides $Q$. Suppose that $x_ld, x_le \not\in E$ for any $l = 1, \dots, \gamma$. This, by considering $\deg A = \deg B$, implies that $\gamma < q$.

\smallskip

\textit{Case 2.1.} Consider the case where there exists $\theta \in \{\gamma+1, \dots, q\}$ such that $L_\theta \subseteq Z = \{z_1, \dots, z_\gamma, d,e\}$. 

If $L_\theta = de$, then $B = (M_1 \cdots M_{q-1}) \cdot L_\theta \cdot ab$, which already appeared before in $\Mc^{(q+1)}$, a contradiction. If, for instance, $L_\theta = z_1d$ (similarly for the case where $L_\theta = z_1 e$), then set
$$C = (M_2 \cdots M_{q-1}) \cdot L_1 \cdot L_\theta \cdot ab.$$
Again, we have $C > B$. Also, $C:B = x_1:e$, so either $B = C$, which contradicts $cd$ being the largest edge in the decomposition of $B$, or $C:B = x_1$ divides $Q$. 

If, for example, $L_\theta = \{z_1,z_2\}$, then without loss of generality assume that $M_{q-1}$ is the largest among $M_3, \dots, M_{q-1}$ and set
$$u_g = (M_3 \cdots M_{q-2}) \cdot L_1 \cdot L_2 \cdot L_\theta \cdot ab.$$
Clearly, $u_g > B/M_{q-1}$ in $\Mc^{(q)}$. Furthermore, $u_g : B/M_{q-1} = x_1x_2: de$ (which is necessarily not 1). It follows that either $u_g = B/M_{q-1}$, which then implies that $B = u_gM_{q-1}$ which contradicts $cd$ being the largest edge in the decomposition of $B$, or there exists a generator $u_h > B/M_{q-1}$ of $I^q$ ($u_h$ could be $u_g$) such that $u_h : B/M_{q-1}$ is either $x_1$ or $x_2$, which then implies that $u_hM_{q-1} : B$ is either $x_1$ or $x_2$ that divides $Q$. This case is completed by showing that $u_hM_{q-1} > B$, which is done in the following Claim.

\smallskip

\noindent\textsc{Claim 2.} $u_hM_{q-1} > B$ in $\Mc^{(q+1)}$.

\noindent\textsc{Proof of Claim 2.} In Claim 1, we have seen a similar statement for a simpler situation. The proof of this claim goes in the same line of arguments (with more details). Since $u_h > B/M_{q-1}$ in $\Mc^{(q)}$, either $u_h$ contains an edge $f$ of $G$ such that $f \succ ad$ or $u_h = u_h'\cdot ad$, where $u_h'$ is a generator of $I^{q-1}$ and $u_h' > B/(M_{q-1}\cdot ad) = M_1 \cdots M_{q-2} \cdot M_q$ in $\Mc^{(q-1)}$. If $u_h$ contains $f \succ ad$, then $u_hM_{q-1} > B$ in $\Mc^{(q+1)}$ as claimed. 

Suppose that $u_h = u_h' \cdot ad$ and $u_h' > M_1 \cdots M_{q-2} \cdot M_q = N$ in $\Mc^{(q-1)}$. To proceed, without loss of generality, we may assume that the relative orders among the edges $M_1, M_2$ and $M_q$ is $M_1 \succeq M_2 \succeq M_q$. 

If $M_{q-1} \succeq M_1$ then $u_h' M_{q-1} > N \cdot M_{q-1}$ in $\Mc^{(q)}$, and so $u_hM_{q-1} = u_h'M_{q-1}\cdot ad > N \cdot M_{q-1} \cdot ad = B$ in $\Mc^{(q+1)}$ as asserted.
 
 If $M_1 \succ M_{q-1} \succeq M_2$, then since $u_h' > N$ in $\Mc^{(q-1)}$, either $u_h'$ contains an edge $f' \succ M_1$ or $u_h' = M_1 \cdot u_h''$, where $u_h'' > M_2 \cdots M_{q-2} \cdot M_q = N'$ in $\Mc^{(q-2)}$. In the former case, when $u_h'$ contains an edge $f' \succ M_1$, we have $u_h'M_{q-1} > NM_{q-1}$ in $\Mc^{(q)}$, and so $u_hM_{q-1} = u_h'M_{q-1} \cdot ad > NM_{q-1} \cdot ad = B$ in $\Mc^{(q+1)}$. In the latter scenario, when $u_h' = M_1 \cdot u_h''$ with $u_h'' > N'$ in $\Mc^{(q-2)}$, we have $u_h'' M_{q-1} > N'M_{q-1}$ in $\Mc^{(q-1)}$, and so $u_h''M_{q-1}M_1 > N'M_{q-1}M_1 = NM_{q-1}$ in $\Mc^{(q)}$. It follows that $u_hM_{q-1} = u_h'\cdot M_{q-1}\cdot ad = u_h'' \cdot M_1 \cdot M_{q-1} \cdot ad > NM_{q-1} \cdot ad = B$ in $\Mc^{(q+1)}$ as desired. 

 If $M_1 \succeq M_2 \succ M_{q-1} \succeq M_q$, then since $u_h' > N$ in $\Mc^{(q-1)}$, either $u_h'$ contains an edge $f' \succ M_1$ or $u_h' = M_1 \cdot u_h''$ with $u_h'' > N'$ in $\Mc^{(q-2)}$. As before, if $u_h'$ contains such an $f'$ then the claim is proved. On the other hand, since $u_h'' > N'$ in $\Mc^{(q-2)}$, either $u_h''$ contains an edge $f'' \succ M_2$ or $u_h'' = M_2 \cdot u_h'''$ with $u_h''' > M_3 \cdots M_{q-2} \cdot M_q = N''$ in $\Mc^{(q-3)}$. If such an edge $f''$ in $u_h''$ exists, then for the same reason as with $f'$, we have $u_h''M_{q-1} > N'M_{q-1}$, and so $u_h' M_{q-1} = u_h''M_{q-1}M_1 > N'M_{q-1}M_1 = NM_{q-1}$, which implies that $u_h M_{q-1} = u_h'M_{q-1} \cdot ad > NM_{q-1} \cdot ad = B$, as asserted. If $u_h'' = M_2\cdot u_h'''$ with $u_h''' > N''$, then $u_h''' M_{q-1} > N''M_{q-1}$, and so $u_h''M_{q-1} = u_h'''M_{q-1}M_2 > N''M_{q-1}M_2 = N'M_{q-1}$. It follows that $u_h'M_{q-1} = u_h''M_{q-1}M_1 > N'M_{q-1}M_1 = NM_{q-1}$ and, therefore, $u_hM_{q-1} = u_h'M_{q-1} \cdot ad > NM_{q-1} \cdot ad = B$ as claimed.
 
Finally, consider the case that $M_q \succ M_{q-1}$. The argument goes similarly as in the previous cases, where we successively write $u_h = u_h'\cdot ad$, $u_h' = u_h'' \cdot M_1$, $u_h'' = u_h''' \cdot M_2$ and $u_h''' = u_h'''' \cdot M_q$ with $u_h'''' > M_3 \cdots M_{q-2}$ in $\Mc^{(q-4)}$; note that $q-4 > 2$ since $q \ge 7$. The argument then completes by tracing back as seen before. The claim is established.

\smallskip

\textit{Case 2.2.} Consider the case where there does not exist such $L_\theta \subseteq Z$. Since $L_{\gamma+1} \cdots L_q$ is contained in $M_{\gamma+1} \cdots M_{q-1} \cdot d \cdot e \cdot z_1 \cdots z_\gamma$, and $\deg (L_{\gamma+1} \cdots L_q) =  \deg(M_{\gamma+1} \cdots M_{q-1}) + 2$, there must exist $\theta, \delta \in \{\gamma+1, \dots, q\}$ and $\sigma \in \{\gamma+1, \dots, q-1\}$ such that:
$L_\theta \cap Z \not= \emptyset, L_\delta \cap Z \not= \emptyset$, and $M_\sigma$ covers (consists of) two vertices, one from each of the edges $L_\theta$ and $L_\delta$. For simplicity of notations, we shall assume that $\sigma = 3$.

Set $L_\theta = uv$ and $L_\delta = u'v'$, where, without loss of generality, $u,u' \in Z$, $v,v' \not\in Z$, and $vv' = M_3$. We shall examine the following possibilities.

(a) $u, u' \in \{z_1, \dots, z_\gamma\}$; for instance, $u = z_1$ and $u' = z_2$. Set 
$$u_g = (L_1L_2L_\theta L_\delta)\cdot (M_4 \cdots M_{q-2}) \cdot ab.$$
Then, $u_g > (M_1 \cdots M_{q-2}) \cdot M_q \cdot ad = B/M_{q-1}$ are generator of $I^q$. Also, $u_g : B/M_{q-1} = x_1x_2 : de$ (which is necessarily not 1). Thus, there exists $u_h > B/M_{q-1}$ in $\Mc^{(q)}$ so that $u_h: B/M_{q-1}$ is a variable that divides $x_1x_2$. By a similar argument as in the Claim, in which we consider the relative order of $M_{q-1}$ compared to $M_1, M_2, M_3$ and $M_q$ noting that $q-5 \ge 2$ since $q \ge 7$, we see that $u_hM_{q-1} > B$. Also, $u_hM_{q-1}: B$ is a variable that divides $Q$ as desired.

(b) $u \in \{z_1, \dots, z_\gamma\}$ and $u' \in \{d,e\}$; for instance, $u = z_1$ and $u' = e$. Set
$$C = (L_1L_\theta L_\delta M_2) \cdot (M_4 \cdot M_{q-1}) \cdot (ab).$$
Then, $C > B$, and $C:B = x_1:d$. Thus, either $B=C$ is redundant, or $C:B = x_1$, which divides $Q$.
\end{proof}

The next two propositions show that linear quotients are implied from $I^2$ to $I^3$, and from $I^3$ to $I^4$.

\begin{Proposition}
    \label{prop.2to3}
    Assume the same setting as in Theorem \ref{thm.LinQuotients}. Suppose that $I^2$ has linear quotients with an order $\Mc^{(2)}$ of its generators. Then, $I^3$ has linear quotients with the order $\Mc^{(3)}$, as constructed as in Definition \ref{def.NewOrder}, of its generators.
\end{Proposition}

\begin{proof}
    We proceed in the same line of arguments as that of Theorem \ref{thm.LinQuotients}. We need to prove that $\Mc^{(3)}$ gives a linear quotient order for $I^3$. Again, consider $A = u_ie_j$ and $B = u_ke_\ell$, where $e_i = ab \succ e_\ell = cd$, $u_i = L_1L_2$, $u_k = M_1M_2$, and $M_t \preceq cd$ for $t = 1,2$.
    Case 1 follows in exactly the same way as that in Theorem \ref{thm.LinQuotients}. 

    In Case 2, if $\gamma < 2$ then $Q$ is a variable and we can choose $C = A$. If $\gamma \ge 2$, then it forces $\gamma = q = 2$. That means 
    \begin{align*}
        A & = (x_1y_1)\cdot (x_2y_2) \cdot (ab) \\
        B & = (y_1z_1) \cdot (be) \cdot (ad).
    \end{align*}
    For $Q = x_1x_2$, we must also have that $y_2$ is $d$ or $e$. That means, either $x_2d$ or $x_2e$ is an edge in $G$. The same argument for this situation in Theorem \ref{thm.LinQuotients} now applies to complete the proof.
    \end{proof}

\begin{Proposition}
    \label{prop.4to5}
    Assume the same setting as in Theorem \ref{thm.LinQuotients}. Suppose that $I^3$ has linear quotients with the order $\Mc^{(3)}$ of its generators. Then, so does $I^4$ with the order $\Mc^{(4)}$ of its generators.
\end{Proposition}

\begin{proof}
    Again, we follow the same line of arguments as in Theorem \ref{thm.LinQuotients} and, at the same time, identify the differences or where arguments in Theorem \ref{thm.LinQuotients} are not applicable.

    Case 1 follows with exactly the same argument. In Case 2, as noted in Proposition \ref{prop.2to3}, we may assume that $\gamma \ge 2$. That means, either $\gamma = 2$ or $\gamma = 3 = q$. 

    If $\gamma = 2$, then
    \begin{align*}
        A & = (x_1y_1)\cdot(x_2y_2)\cdot L_3 \cdot (ab) \\
        B & = (y_1z_1)\cdot(y_2z_2)\cdot (be)\cdot (ad).
    \end{align*}
    Since $Q = x_1x_2$, this forces $L_3 \subseteq \{z_1,z_2,d,e\}$, in which case the same argument in Theorem \ref{thm.LinQuotients} (when there exists $\theta$ such that $L_\theta \subseteq Z$) applies.

    If $\gamma = 3$, then 
    \begin{align*}
        A & = (x_1y_1)\cdot (x_2y_2) \cdot (x_3y_3) \cdot (ab) \\
        B & = (y_1z_1)\cdot (y_2z_2) \cdot (be) \cdot (ad),
    \end{align*}
    Since $Q = x_1x_2x_3$, this implies that $y_3$ is either $d$ or $e$. That is, either $x_3d$ or $x_3e$ is an edge in $G$. The same argument of Theorem \ref{thm.LinQuotients} in this case also applies to complete the proof. 
\end{proof}

\begin{Remark}
    Our proof does not currently work to show that $I^k$ has linear quotients implies $I^{k+1}$ has linear quotients for $k=4,5,6$
    in general.
\end{Remark}

\begin{Remark}
For relatively manageable graphs, the techniques from the proof of Theorem \ref{thm.LinQuotients} can be used to show that $I^k$ has linear quotients implies $I^{k+1}$ has linear quotients for $k=4,5,6$; 
for instance, if $G$ is an  \emph{anticycle}, i.e., the complementary graph of a cycle. In particular, by exhibiting a specific linear quotients order for $I(G)^2$, when $G$ is an anticycle, it follows that $I(G)^q$ has linear quotients for all $q \ge 2$. This is a main result of a recent preprint \cite{B+}.
\end{Remark}

\section*{Acknowledgements}
The research for this paper was initiated during the authors' stay at the American Institute of Mathematics (AIM), as part of the AIM SQuaRE program.  We would like to thank AIM for the warm hospitality and providing a stimulating  research environment.  Faridi is supported by NSERC Discovery Grant 2023-05929. H\`a  is partially supported by a Simons Foundation grant.  Kara is supported by NSF grant DMS-2418805.

\end{document}